\documentclass[12pt]{article} 

\usepackage[utf8]{inputenc}
\usepackage[T1]{fontenc}
\usepackage{appendix}
\usepackage{longtable}
\usepackage{amsmath}
\usepackage{amsthm}
\usepackage{amsfonts}
\usepackage{amssymb}
\usepackage{breqn}
\usepackage{xcolor}
\usepackage{xifthen}
\usepackage{parskip}
\usepackage{verbatim}
\usepackage[style=numeric, url=false]{biblatex}
\usepackage[left=2.5cm, right=2.5cm,bottom=4.5cm]{geometry}
\allowdisplaybreaks

\newcommand{\Ebb}{\mathbb{E}}
\newcommand{\Fbb}{\mathbb{F}}

\newcommand{\Nbb}{\mathbb{N}}

\newcommand{\Rbb}{\mathbb{R}}

\newcommand{\CA}{\mathcal{A}}

\newcommand{\CD}{\mathcal{D}}

\newcommand{\CF}{\mathcal{F}}

\newcommand{\CI}{\mathcal{I}}

\newcommand{\CK}{\mathcal{K}}

\newcommand{\CP}{\mathcal{P}}
\newcommand{\CQ}{\mathcal{Q}}

\newtheorem{prp}{Proposition}[section]
\newtheorem{thm}[prp]{Theorem}
\newtheorem{lem}[prp]{Lemma}
\newtheorem{ass}[prp]{Assumption}
\newtheorem{cor}[prp]{Corollary}
\theoremstyle{definition}

\newtheorem{rmk}[prp]{Remark}
\numberwithin{equation}{section}

\let\[\relax
\let\]\relax

\let\d\relax
\DeclareMathOperator{\d}{d\!}

\DeclareMathOperator{\ind}{\mathbf{1}}

\renewcommand{\l}{\left}
\renewcommand{\r}{\right}

\newcommand\numberthis{\addtocounter{equation}{1}\tag{\theequation}}

\renewcommand{\theta}{\vartheta}
\renewcommand{\phi}{\varphi}
\renewcommand{\epsilon}{\varepsilon}

\DeclareMathOperator{\E}{\hat{\Ebb}}
\newcommand{\M}{\textnormal{M}_{\ast}}
\newcommand{\MM}{\overline{\textnormal{M}}_{\ast}}
\renewcommand{\H}{\textnormal{H}_{\ast}}
\newcommand{\F}[1]{\textnormal{B}_b(\Omega_{#1})}
\renewcommand{\L}{\textnormal{L}_{\ast}}
\renewcommand{\S}{\textnormal{M}_b}

\title{Mean-Field SDEs driven by $G$-Brownian Motion}
\author{Karl-Wilhelm Georg Bollweg, Thilo Meyer-Brandis}

\begin{document}

    \maketitle
    \section{Introduction}\label{sec:introduction}
    A mean-field SDE driven by Brownian motion is an equation of the form
    \begin{equation*}
        \d X_t = b(X_t, P_{X_t}) \d t + \sigma(X_t,P_{X_t}) \d B_t,
    \end{equation*}
    where $B$ denotes Brownian motion and $P_{X_t}=P\circ (X_t)^{-1}$ denotes the push-forward measure of $X_t$ for $t\geq 0$.
    Mean-field theory can be traced back to the pioneering work of McKean \cite{mckean_class_1966} and Vlasov \cite{vlasov_vibrational_1968}, who studied the mean-field limit for interacting particle systems, revealing the connection between microscopic and macroscopic descriptions.
    The underlying concepts, deeply rooted in statistical physics, have found application in a wide range of disciplines, from physics to economics and biology, providing a fundamental understanding of systems at the macroscopic level.
    Mean-field models are used for animal and cell population behaviour, infectious diseases, systemic risk, and fluid mechanics, to name but a few.
    In 2006, Lasry and Lions brought a game-theoretic perspective to mean-field interactions and introduced so called mean-field games, cf. \cite{lasry_jeux_2006}, \cite{lasry_jeux_2006-1}, \cite{lasry_mean_2007}.
    The theory of mean-field games has paved the way for addressing dynamic interactions among weakly dependent rational agents in strategic decision-making scenarios.
    Thus, it sheds light on strategic behavior in large populations and offers a powerful approach for analysing systems with decentralised decision-making.
    Further contributions, e.g. \cite{caines_large_2006}, \cite{carmona_mean_2016}, have extended the reach of mean-field games.
    In the last decades, mean-field and mean-field game theory have been very active areas of research as reflected by the great volume of literature on these topics.  

    Another area of research that has attracted great interest in recent years is Knightian or model uncertainty.
    Model uncertainty describes situations where the underlying probability measure is not known.
    Thus, uncertainty affects any results or conclusions obtained using that model, and the need for robust modelling approaches arises.
    However, classical probability theory is not designed to include uncertainty.
    In the last decades, various approaches to quantify uncertainty have been developed. Most notable is the literature on robust finance, e.g. \cite{nutz_robust_2015}, \cite{pham_portfolio_2022}, \cite{neufeld_robust_2016}, \cite{lutkebohmert_robust_2021}, \cite{liang_robust_2020} and imprecise probabilities, e.g. \cite{quaeghebeur_introduction_2022}, \cite{augustin_impact_2010}, \cite{augustin_statistics_2022}. 
    The aim of this paper is to unify mean-field theory and uncertainty quantification, and to contribute to the development of a mean-field theory under model uncertainty.

    One main approach to study stochastic processes under model uncertainty is to consider the so-called $G$-Brownian motion as underlying driving process.
    Peng introduced $G$-Brownian motion in his seminal works on sublinear expectation spaces in the early 2000s, cf. \cite{peng_multi-dimensional_2008}, \cite{peng_new_2008}, \cite{peng_nonlinear_2019}.
    There, the classical probability space $(\Omega,\CF,P)$ is replaced by a sublinear expectation space $(\Omega,\CF,\E)$.
    In this setting, $G$-Brownian motion can be regarded as Brownian motion with volatility uncertainty.

    A second and more recent approach to stochastic processes under uncertainty is to consider a possibly uncountable set of probability measures $\mathcal{P}$ such that the canonical process $X$ has the desired properties. This is achieved by specifying a set $\Theta(t,\omega)$ of possible local semimartingale characteristics, i.e., each $P\in \CP$ is such that
    \begin{align*}
        \l(b^P_t(\omega), a^P_t(\omega), k^P_t(\omega)\r) & \in \Theta(t,\omega) & \d t\otimes \d P\text{-a.e.,} \numberthis\label{Theta}
    \end{align*}
    where $(b^P,a^P,k^P)$ denote the local semimartingale characteristics of the canonical process under $P$. In particular, the case where $\mathcal P$ is non-dominated and there is no preference among the measures in $\mathcal P$ is of great interest but it also poses major technical challenges. Within this framework, $G$-Brownian motion can be associated with a certain non-empty, convex set of probability measures $\CP$ representing uncertainty about the volatility, but the approach also allows for more general uncertainty modeling such as Lévy processes (\cite{neufeld_nonlinear_2016}) and affine processes (\cite{fadina_affine_2019}, \cite{biagini_non-linear_2023}) under uncertainty.
    
    For mean-field type processes, the local semimartingale characteristics not only depend on the current state of the process but also on its distribution. Thus, the set of possible characteristics $\Theta$ in \eqref{Theta} would depend not only on $(t,\omega)$ but also on $P$. This makes the uncertainty set approach significantly more involved for mean-filed processes than for Lévy or affine models. In particular, in an ongoing research project we examine the question whether such a set $\Theta$ could satisfy Assumption 2.1 in \cite{nutz_constructing_2013} that implies the crucial dynamic programming principle.
        
    In this paper, we focus on the introduction of mean-field SDEs under volatility uncertainty. Considering $G$-Brownian motion as underlying driving process allows us to draw from the comparatively rich mathematical structure developed in the $G$-framework. A first step towards defining mean-field SDEs in the $G$-setting was made in \cite{sun_mean-field_2020}. 
    There, the author considered dynamics of the form
    \begin{align*}
        \d X_s & = \E\!\l[ b\!\l(s,x,X_s \r)\r]\Big|_{x=X_s} \d s + \E\!\l[ h\!\l(s,x,X_s \r)\r]\Big|_{x=X_s} \d \l<B\r>_s + \E\!\l[ g\!\l(s,x,X_s \r)\r]\Big|_{x=X_s} \d B_s,
        \numberthis \label{eq:intro-1}
    \end{align*}
    where $B$ is $G$-Brownian motion and $\E$ the corresponding $G$-expectation.
    They proved the existence of a unique solution of \eqref{eq:intro-1} under a Lipschitz condition on the coefficients.
    In their recent work \cite{sun_distribution_2023}, the authors generalised these results to Lipschitz coefficients that depend on the sublinear distribution of the solution.
    That is, they studied dynamics of the form
    \begin{equation*}
        \d X_s = b\!\l(s,x,F_{X_s} \r)\Big|_{x=X_s} \d s + h\!\l(s,x,F_{X_s} \r)\Big|_{x=X_s} \d \l<B\r>_s + g\!\l(s,x,F_{X_s} \r)\Big|_{x=X_s} \d B_s, \numberthis\label{eq:intro-11}
    \end{equation*}
    where for a random variable $\xi$ the functional $F_\xi$ is given by $F_\xi(\phi):=\E\!\l[\phi(\xi)\r]$, where $\phi:\,\Rbb^d\rightarrow\Rbb$ lies in a suitable function space. 
    This corresponds to the sublinear distribution of $\xi$, cf. \cite{peng_nonlinear_2019}.
    For this purpose, the authors constructed a space containing all sublinear distribution functions endowed with a metric which can be regarded as a generalisation of the Kantorovich-Rubinstein metric. 
    
    In this paper, we introduce a different, more general formulation of mean-field $G$-SDEs that has several advantages as discussed below. 
    Instead of letting the coefficients depend on the sublinear distribution $F_{X_s}$, we allow the coefficients to depend directly on the random variable $X_s$.
    To be specific, we work in the generalised $G$-setting introduced in Chapter~8 in \cite{peng_nonlinear_2019} and consider $G$-SDEs of the form
    \begin{align*}
        \d X_s(\omega) &=b\!\l(s,x, X_s,\omega\r)\Big|_{x=X_s(\omega)} \d s + h\!\l(s,x, X_s,\omega\r)\Big|_{x=X_s(\omega)} \d \l<B\r>_s(\omega) 
        \\&\qquad + g\!\l(s,x, X_s,\omega \r)\Big|_{x=X_s(\omega)} \d B_s(\omega),
        \numberthis\label{eq:intro-2}
    \end{align*} 
    where the coefficients are defined on $[0,T]\times\Rbb^d\times\L^{2,d}\times\Omega$.
    Here, $\L^{2,d}$ denotes the space of $d$-dimensional random vectors with finite second moment, cf. Section~\ref{sec:G-setting}.
    For simplicity and conciseness, we will use the following notation. 
    For a function $f$ on $[0,T]\times\Rbb^d\times\L^{2,d}\times\Omega$ and random vector $\eta\in\L^{1,d}$, define
    \begin{equation*}
        f\!\l(s,\eta,\xi,\omega\r):=f\!\l(s,\eta(\omega),\xi,\omega\r)=f\!\l(s,x,\xi,\omega\r)\Big|_{x=\eta(\omega)}
    \end{equation*}
    for any $0\leq s\leq T$, $\xi\in\L^{2,d}$ and $\omega\in\Omega$. 
    Often we suppress the explicit dependence on $\omega$ and write $f\!\l(s,\eta,\xi \r)$ instead of $f\!\l(s,\eta,\xi,\omega\r)$. 
    Hence, the $G$-SDE \eqref{eq:intro-2} can be written as
    \begin{align*}
        \d X_s &=b\!\l(s,X_s, X_s\r) \d s + h\!\l(s,X_s, X_s\r) \d \l<B\r>_s + g\!\l(s,X_s, X_s\r) \d B_s
        .\numberthis\label{eq:intro-22}
    \end{align*}

    Our approach to let the coefficients depend on the random variable $X_s$ instead of its distribution is inspired by \cite{lions_large_2007}, \cite{buckdahn_mean-field_2017}, \cite{carmona_mean_2016}. 
    There, the authors lifted functions on probability measures to functions on random variables which allowed the definition of the Lions derivative via the Fréchet derivative of the lifted function. 
    Analogously, our formulation enables the study of the Fréchet differentiability of solutions of \eqref{eq:intro-22}.
    This will be used in the companion paper \cite{bollweg_mean-field_nodate} to associate \eqref{eq:intro-22} with a nonlocal nonlinear PDE.
    This Feynman-Kac type result allows the computation of functionals of the solution of \eqref{eq:intro-22} by solving the associated PDE.
    Our formulation also enables the introduction of a corresponding finite interacting particle system such that mean-field $G$-SDE \eqref{eq:intro-22} can be regarded as the asymptotic limit of this system. This result corresponds to propagation of chaos in classical mean-field theory, see \cite{criens_set-valued_2023} for a set valued approach.

    The main contribution of this paper is to study existence and uniqueness of solutions of the mean-field $G$-SDE \eqref{eq:intro-22}.
    We remark that coefficients of the form in \eqref{eq:intro-1} and \eqref{eq:intro-11} are special cases of the coefficients of the form \eqref{eq:intro-22}.
    Further, in contrast to \cite{sun_mean-field_2020} and \cite{sun_distribution_2023}, we allow the coefficients to be non-deterministic and non-Lipschitz, and consider square integrable instead of deterministic initial conditions.
    More precisely, we mainly require the coefficients to satisfy an Osgood type continuity and sublinear growth condition, cf. Assumption~\ref{ass:1-non-lipschitz}.  
    Under these assumptions, we derive the existence and uniqueness result for \eqref{eq:intro-22} as formalised in Theorem~\ref{thm:existence-uniqueness-ass-1}.
    To be specific, for a given square integrable initial condition $X_t=\xi\in\L^{2,d}(t)$, $0\leq t\leq T$, we construct the solution of \eqref{eq:intro-2} as the limit of a Picard sequence.
    Furthermore, we show that this solution is unique and satisfies
    \begin{equation*}
        \E\!\l[ \sup_{t\leq s\leq T} \l\|X_s\r\|^2\r]<\infty.
    \end{equation*}
    This is in line with results for classical mean-field SDEs.
    We are not aware of such results for $G$-SDEs under weaker assumptions on the coefficients.
    In particular, since coefficients such as in \cite{sun_distribution_2023}, \cite{sun_mean-field_2020} are a special case of the coefficients considered in this paper, the existence and uniqueness results in \cite{sun_distribution_2023}, \cite{sun_mean-field_2020} follow immediately from Theorem~\ref{thm:existence-uniqueness-ass-1}.
    The comparison of our results with the existing literature will be illustrated in more detail in Section~\ref{sec:discussion}.

    Like most of the literature on $G$-SDEs, \cite{sun_mean-field_2020}, \cite{sun_distribution_2023} work in the $G$-setting as constructed in \cite{peng_multi-dimensional_2008} which is restricted to quasi-continuous random variables, cf. \cite{denis_function_2011}.
    We work in the generalised $G$-setting from Chapter 8 in \cite{peng_nonlinear_2019} and extend many known results from the quasi-continuous to this more general setting.
    For the sake of completeness, we include the proofs of these results in Appendix~\ref{app:proofs_gsetting}.

    Our paper is structured as follows.
    In Section~\ref{sec:G-setting}, we recall the sublinear expectation setting from \cite{peng_nonlinear_2019}. 
    In Section~\ref{sec:existence-uniqueness}, we introduce our mean-field $G$-SDE and prove existence and uniqueness of its solution.
    Finally, we compare our results to the existing literature in Section~\ref{sec:discussion}.

    \section{The $G$-Setting}\label{sec:G-setting}
    In this section, we recall the sublinear expectation setting from Chapter~8 in \cite{peng_nonlinear_2019}.
    This setting builds on the results in quasi-sure analysis related to the $G$-setting, cf. \cite{denis_function_2011}.
    If not denoted otherwise, the proofs can be found in Appendix~\ref{app:proofs_gsetting}.
    
    For $n\geq 1$, let $\Omega:=\textnormal{C}_0(\Rbb_+,\Rbb^n)$ be the space of all continuous $\Rbb^n$-valued paths starting at the origin equipped with the topology of uniform convergence 
    .
    Let $\CF$ denote the corresponding Borel $\sigma$-algebra.
    Moreover, let $\Fbb=(\CF_t)_{t\geq 0}$ denote the natural filtration generated by the coordinate mapping process $B$.    
    For $t\geq 0$, let $\F{t}$ denote the space of all bounded $\CF_t$-measurable functions $\xi:\,\Omega\rightarrow \Rbb$.

    Fix a convex and closed subset $\Sigma\subseteq \mathbb{S}_+^n$ of symmetric non-negative definite $n\times n$-matrices and set
    \begin{equation*}
        \mathcal{A}^\Sigma:=\Big\{ \theta=(\theta_t)_{t\geq 0}\,:\,\theta \text{ is $\Sigma$-valued and $\mathbb{F}$-progressively measurable}\Big\}.
    \end{equation*}
    Let $P_0$ denote the Wiener measure on $(\Omega,\CF)$, and define
    \begin{equation*}
        \CP:=\l\{ P_0 \circ \l( \theta \bullet B\r)^{-1}\,:\, \theta \in\CA^{\Sigma}\r\},
    \end{equation*}
    where $\theta\bullet B:=\int_0^\cdot \theta_s \d B_s$ denotes the It\^o integral with respect to the stochastic basis $(\Omega,\CF,\Fbb,P_0)$.
    The set of probability measures $\CP$ induces an upper expectation on $\F{}$, namely
    \begin{equation*}
        \E:\,\F{} \rightarrow \Rbb,\qquad \xi\mapsto\E\!\l[ \xi \r]:=\sup_{P\in\CP} E_P\l[\xi\r],
    \end{equation*}
    where $E_P$ denotes the linear expectation with respect to $P$.
    The process $B$ is a $G$-Brownian motion with respect to $\E$ and $(\Omega,\F{},\E)$ is a sublinear expectation space.
    
    For $p\geq 1$, define the norm
    \begin{equation*}
        \|\cdot\|_{\L^{p}}:\,\F{}\rightarrow\Rbb_+,\qquad \xi\mapsto\|\xi\|_{\L^{p}}:=\E\!\l[\l| \xi\r|^p\r]^{\frac{1}{p}}.
    \end{equation*}
    For any $t\geq 0$, let $\L^{p}$ and $\L^{p}(t)$ denote the completion of $\F{}$ and $\F{t}$ with respect to $\|\cdot\|_{\L^{p}}$ respectively.

    Let $\S(0,T)$ denote the space of all processes $X:\,[0,T]\times\Omega\rightarrow \Rbb$ of the form
    \begin{equation*}
        X_s(\omega)=\sum_{k=0}^{m-1} \xi_k(\omega)\ind_{\l[t_k,t_{k+1}\r)}(s),\qquad (s,\omega)\in [0,T]\times\Omega,
    \end{equation*}
    with $m\in\mathbb{N}$, $0=t_0<\ldots<t_m=T$, and $\xi_k\in\F{t_k}$ for all $0\leq k\leq m-1$.
    For $p\geq 1$, define the norm
    \begin{align*}
        \|\cdot\|_{\M^{p}}:\,\S(0,T)\rightarrow\Rbb_+,\qquad \|X\|_{\M^{p}}:= \l(\int_0^T \E\!\l[  \l| X_s\r|^p  \r] \d s\r)^{\frac{1}{p}}
    \end{align*}
    and let $\M^{p}(0,T)$ denote the completion of $\S(0,T)$ with respect to $\|\cdot\|_{\M^{p}}$.
    
    \begin{rmk}
        Sometimes, it makes sense to consider the norm
        \begin{equation*}
            \|\cdot\|_{\MM^{p}}:\,\S(0,T)\rightarrow\Rbb_+,\qquad \|X\|_{\MM^{p}}:= \E\!\l[ \int_0^T \l| X_s\r|^p  \d s \r]^{\frac{1}{p}},
        \end{equation*}
        which is weaker than $\l\|\cdot\r\|_{\M^p}$, and thus $\M^p(0,T)\subseteq\MM^p(0,T)$, where $\MM^p(0,T)$ denotes the completion of $\S(0,T)$ with respect to $\|\cdot\|_{\MM^p}$.
        For the study of $G$-SDEs, the spaces $\M^p(0,T)$ are more convenient and thus we focus on these in this paper.
    \end{rmk}

    \begin{lem}\label{lem:Lp-sum-Mp}
        Let $p\geq 1$, $N\in\mathbb{N}$, $0=t_0<\ldots<t_N=T$, and $\xi_k\in\L^p(t_k)$, $0\leq k\leq N-1$.
        Then
        \begin{equation*}
            X:=\sum_{k=0}^{N-1} \xi_k\ind_{\l[t_k,t_{k+1}\r)}\in\M^p(0,T).
        \end{equation*}
    \end{lem}
    
    \begin{lem}\label{lem:B-M2}
        Let $a\in\Rbb^n$. Then $B^a:=a^T B \in\M^2(0,T)$.
    \end{lem}
    
    For $a\in\Rbb^n$, define the map $\CI_a:\,\S(0,T)\rightarrow \L^{2}(T)$ by
    \begin{equation*}
        \CI_a(X):=\int_0^T X_t \d B^a_t := \sum_{k=0}^{m-1} \xi_k \l(B^a_{t_{k+1}} -B^a_{t_k}\r)
    \end{equation*}
    for each
    \begin{equation*}
        X=\sum_{k=0}^{m-1} \xi_k\ind_{\l[t_k,t_{k+1}\r)}\in\S(0,T).
    \end{equation*}
    \begin{lem}\label{lem:I-properties}
        For $a\in\Rbb^n$, the map $\CI_a:\,\S(0,T)\rightarrow \L^{2}(T)$ is continuous and linear, thus, it can be continuously extended to $\CI_a:\,\M^{2}(0,T)\rightarrow \L^{2}(T)$.
        Moreover, we have for all $X\in\M^{2}(0,T)$ that
        \begin{equation*}
            \E\l[\CI_a(X)\r] = \E\l[ \int_0^T X_t \d B^a_t \r] =0,
        \end{equation*}
        \begin{equation*}
            \l\|\CI_a(X)\r\|_{\L^2}^2 = \E\l[ \l(\int_0^T X_t \d B^a_t\r)^2 \r] 
            \leq 
            \overline{\sigma}_{aa}^2 \E\l[\int_0^T \l| X_t \r|^2 \d t \r]
            \leq
            \overline{\sigma}_{aa}^2 \l\|X\r\|_{\M^2}^2
            ,
        \end{equation*}
        where
        \begin{equation*}
            \overline{\sigma}_{aa}:=\sup_{\sigma\in\Sigma} \sqrt{a^T \sigma \sigma a}.
        \end{equation*}
    \end{lem}
    \begin{proof}
        Cf. \cite[Lemma 8.1.6]{peng_nonlinear_2019}.
    \end{proof}
    For $0\leq t \leq s \leq  T$ and $X\in\M^{2}(0,T)$, set
    \begin{align*}
        \int_t^s X_u\d B^a_u &:=\int_0^T X_u\ind_{[t,s)}(u) \d B^a_u=\CI_a(X \ind_{[t,s)}).
    \end{align*}

    By Lemma~\ref{lem:B-M2}, we can define the quadratic variation of $B^a$ as process $\l<B^a\r>$ given by
    \begin{equation*}
        \l<B^a\r>_t:= \l(B^a_t\r)^2 -2 \int_0^t B^a_s \d B^a_s \in\L^1(t), \qquad 0\leq t\leq T.
    \end{equation*}

    Now, define the map $\CQ_a:\,\S(0,T)\rightarrow \L^{1}(T)$ by
    \begin{equation*}
        \CQ_a(X)=\int_0^T X_t \d \l< B^a\r>_t := \sum_{k=0}^{m-1} \xi_k\,\l(\l< B^a\r>_{t_{k+1}} - \l< B^a\r>_{t_k}\r)
    \end{equation*}
    for each
    \begin{equation*}
        X=\sum_{k=0}^{m-1} \xi_k\ind_{\l[t_k,t_{k+1}\r)}\in\S(0,T).
    \end{equation*}
    
    \begin{lem}\label{lem:Q-properties}
        For $a\in\Rbb^n$, the map $\CQ_a:\,\S(0,T)\rightarrow \L^{2}(T)$ is continuous and linear, thus, it can be continuously extended to $\CQ_a:\,\M^{1}(0,T)\rightarrow \L^{2}(T)$.
        Moreover, we have for all $X\in\M^{1}(0,T)$ that
        \begin{align*}
            \l\|\CQ_a(X)\r\|_{\L^1}=\E\l[ \l| \int_0^T X_t \d \l<B^a\r>_t \r| \r] &\leq \overline{\sigma}_{aa}^2 \E\l[ \int_0^T \l|X_t\r| \d t \r]\leq \overline{\sigma}_{aa}^2 \l\|X\r\|_{\M^1}.
        \end{align*}
    \end{lem}
    \begin{proof}
        Cf., \cite[Proposition 8.1.10]{peng_nonlinear_2019}.
    \end{proof}
    
    For $0\leq t \leq s \leq  T$ and $X\in\M^1(0,T)$, set
    \begin{align*}
        \int_t^s X_u\d \l<B^a\r>_u &:=\int_0^T X_u\ind_{[t,s)}(u) \d \l<B^a\r>_u=\CQ_a(X\ind_{[t,s)}).
    \end{align*}
    
    \begin{lem}\label{lem:I:M2-M2-map}
        Let $a\in\Rbb^n$ and $X\in\M^{2}(0,T)$. 
        Define $Z$ by
        \begin{equation*}
            Z_s:=\int_0^s X_u \d B^a_u=\CI_a\l(X\ind_{[0,s)}\r), \qquad 0\leq s\leq T.
        \end{equation*}
        Then $Z\in\M^{2}(0,T)$.
    \end{lem}
    
    \begin{lem}\label{lem:QV-M1}
        For $a\in\Rbb^n$, we have $\l<B^a\r>\in\M^1(0,T)$.
    \end{lem}
    \begin{proof}
        Since $B^a\in\M^2(0,T)$, we have $(B^a)^2\in\M^1(0,T)$ and $\int_0^\cdot B^a_s \d B^a_s\in\M^2(0,T)\subseteq \M^1(0,T)$ due to Lemma~\ref{lem:I:M2-M2-map}.
        Hence, $\l<B^a\r>\in\M^1(0,T)$ as sum of elements in $\M^1(0,T)$.
    \end{proof}
    
    Analogous to Lemma~\ref{lem:I:M2-M2-map}, we have the following result.
    \begin{lem}\label{lem:Q:M1-M1-map}
        Let $a\in\Rbb^n$ and $X\in\M^{1}(0,T)$. 
        Define $Z$ by
        \begin{equation*}
            Z_s:=\int_0^s X_u \d \l< B^a \r>_u=\CQ_a\l(X\ind_{[0,s)}\r), \qquad 0\leq s\leq T.
        \end{equation*}
        Then $Z\in\M^{1}(0,T)$.
    \end{lem}

    \begin{lem}\label{lem:t-Mp}
        Let $p\geq 1$, then $\textnormal{id}_{[0,T]}\in\M^p(0,T)$.
    \end{lem}
    \begin{proof}
        For $m>0$, set $t^m_k:=\frac{k\,T}{m}$ and define
        \begin{equation*}
            f^m:=\sum_{k=0}^{m-1}t_k\ind_{[t_k,t_{k+1})}\in\S(0,T).
        \end{equation*}
        Then
        \begin{equation*}
            \int_0^T E\l[\l|f^m(s) - s \r|^p \r] \d s
            = \sum_{k=0}^{m-1} \int_{t_k}^{t_{k+1}} \l| t_k -s \r|^p \d s 
            \leq \sum_{k=0}^{m-1} \sup_{t_k\leq s\leq t_{k+1}} \l| t_k -s \r|^p \frac{T}{m}
            = \frac{T^{p+1}}{m^{p}}.
        \end{equation*}
        The right-hand side vanishes when $m\rightarrow\infty$. 
        Hence, we deduce that $\textnormal{id}_{[0,T]}\in\M^p(0,T)$.
    \end{proof}
    
    \begin{lem}\label{lem:int:M1-M1-map}
        Let $X\in\M^1(0,T)$, and define $Z$ by
        \begin{equation*}
            Z_s:=\int_0^s X_u \d u, \qquad 0\leq s\leq T.
        \end{equation*}
        Then $Z\in\M^1(0,T)$.
    \end{lem}
    
    For $a,b\in\Rbb^n$, we define the quadratic co-variation of $B^a$ and $B^b$ by
    \begin{equation*}
        \l<B^a,B^b\r>:=\frac{1}{4}\l( \l<B^{a+b}\r> - \l<B^{a-b}\r>\r),
    \end{equation*}
    and, for $X\in\M^{1}(0,T)$, we define
    \begin{equation*}
        \int_t^s X_u \d\l<B^a,B^b\r>_u := \frac{1}{4}\l( \CQ_{a+b}(X\ind_{[t,s)}) - \CQ_{a-b}(X\ind_{[t,s)}) \r).
    \end{equation*}

    \begin{lem}\label{lem:Q-p-bound}
        Let $a,b\in\Rbb^n$, $p\geq 1$, $X\in\M^p(0,T)$ and $0\leq t\leq s\leq T$.
        Then
        \begin{equation*}
            \E\l[ \sup_{t\leq w\leq s}\l|\int_t^w X_u \d \l<B^a,B^b\r>_u\r|^p \r] \leq \overline{\sigma}_{ab}^{2p} \l(s-t\r)^{p-1} \int_t^s \E\l[ \l|X_u \r|^p \r] \d u,
        \end{equation*}
        where
        \begin{equation*}
            \overline{\sigma}_{ab}:=\sup_{\sigma\in\Sigma} \sqrt{\l| a \sigma\sigma b \r|}.
        \end{equation*}
    \end{lem}
    
    \begin{lem}\label{lem:I-p-bound}
        Let $a\in\Rbb^n$, $p\geq 2$, $X\in\M^p(0,T)$ and $0\leq t\leq s\leq T$.
        Then
        \begin{equation*}
            \E\l[ \sup_{t\leq w\leq s} \l|\int_t^w X_u \d B^a_u\r|^p \r] \leq C_p\,\overline{\sigma}^{p}_{aa} \l(s-t\r)^{\frac{p-2}{2}} \int_t^s \E\l[ \l|X_u \r|^p \r] \d u,
        \end{equation*}
        where $C_p>0$ is the constant from the Burkholder-Davis-Gundy inequality.
    \end{lem}
    
    \section{Existence and Uniqueness Results}\label{sec:existence-uniqueness}
    
    In this section, we establish existence and uniqueness results for mean-field $G$-SDEs as introduced in Section~\ref{sec:introduction}.
    To be precise, we fix a finite time horizon $0<T<\infty$ and a dimension $d\geq1$, and consider coefficients which are defined on $[0,T]\times\Rbb^d\times\L^{2,d}\times\Omega$, where $\L^{2,d}$ denotes the space of $d$-dimensional random vectors with components in $\L^2$. 
    That is, $\L^{2,d}:=(\L^2)^d$ and $\L^{2,d}(t):=(\L^2(t))^d$ for $0\leq t\leq T$.
    
    
    We fix an initial time $0\leq t\leq T$, and consider the $G$-SDE 
    \begin{align*}
        \d X_s &=b\!\l(s,X_s, X_s\r) \d s + h\!\l(s,X_s, X_s\r) \d \l<B\r>_s + g\!\l(s,X_s, X_s\r) \d B_s, \qquad t\leq s\leq T, 
        \numberthis \label{eq:MF-GSDE}
    \end{align*}
    where $b:\,[0,T]\times\Rbb^d\times\L^{2,d}\times\Omega\rightarrow \Rbb^d$, $h:\,[0,T]\times\Rbb^d\times\L^{2,d}\times\Omega\rightarrow \Rbb^{d\times n\times n}$ and $g:\,[0,T]\times\Rbb^d\times\L^{2,d}\times\Omega\rightarrow \Rbb^{d\times n}$ are given.
    
    Let $\M^{p,d}(t,T)$ denote the space of all $d$-dimensional processes $X:\,[t,T]\times \Omega\rightarrow \Rbb^d$ such that $X=(X^1,\ldots,X^d)^T$ and $X^k\ind_{[t,T)}\in\M^p(0,T)$ for $1\leq k\leq d$.
    A solution of \eqref{eq:MF-GSDE} is a  process $X\in\M^{2,d}(t,T)$ such that its components $X^k$, $1\leq k\leq d$ satisfy
    \begin{align*}
        \d X^k_s &=b_k\!\l(s,X_s, X_s\r) \d s + \sum_{i,j=1}^n h_{kij}\!\l(s,X_s, X_s\r) \d \l<B^i,B^j\r>_s + \sum_{i=1}^n g_{ki}\!\l(s,X_s, X_s\r) \d B_s, \; t\leq s\leq T,
    \end{align*}
    where $b_k,h_{kij},g_{ki}$, $1\leq i,j\leq n$, $1\leq k\leq d$ denote the components of the coefficients.

    \begin{ass}\label{ass:1-non-lipschitz}
        The coefficients $b:\,[0,T]\times\Rbb^d\times\L^{2,d}\times\Omega\rightarrow \Rbb^d$, $h:\,[0,T]\times\Rbb^d\times\L^{2,d}\times\Omega\rightarrow \Rbb^{d\times n\times n}$, and $g:\,[0,T]\times\Rbb^d\times\L^{2,d}\times\Omega\rightarrow \Rbb^{d\times n}$ are such that the following holds for all components $f=b_k,h_{kij},g_{ki}$, $1\leq i,j\leq n$, $1\leq k\leq d$.
        \begin{enumerate}
            \item $f(\cdot,x,\xi)\ind_{[s,T]}\in\M^1(0,T)$ for all $x\in\Rbb^d$, $\xi\in\F{s}^d$ and $t\leq s\leq T$.
            \item There exist an integrable function $\kappa:\,[0,T]\rightarrow\Rbb_+$, a process $K\in\M^1(0,T)$, and continuous, increasing and concave functions $\rho_1, \rho_2:\,\Rbb_+\rightarrow\Rbb_+$ with $\rho_1(0)=\rho_2(0)=0$ and
        \begin{equation}\label{eq:ass-1-rho-integral}
            \int_{0}^1 \frac{1}{\rho_1(r)+\rho_2(r)} \d r = +\infty,
        \end{equation}
        such that
        \begin{align}
            \l| f\l(s,x,\xi,\omega\r) - f\l(s,y,\eta,\omega\r) \r|^2
            & \leq \kappa(s)\,\rho_1\!\l( \l\|x-y\r\|^2 \r)+K_s(\omega)\,\rho_2\!\l(\l\|\xi-\eta\r\|_{\L^{2}}^2\r), 
            \label{ineq:ass-1-continuity} \\
            \l| f\l(s,x,\xi,\omega\r) \r|^2 
            & \leq \kappa(s)\l\|x\r\|^2 + K_s(\omega)\l(1 + \l\|\xi\r\|_{\L^{2}}^2\r)
            \label{ineq:ass-1-growth}
        \end{align}
        for all $\omega\in\Omega$, $t\leq s\leq T$, $x,y\in\Rbb^d$, and $\xi,\eta\in\L^{2,d}(T)$.
        \end{enumerate}
    \end{ass}
    For convenience, define the constant $C:=\int_0^T \kappa(s) \d s + \|K\|_{\M^1}$.
    
    \begin{lem}\label{lem:integrable-coefficients}
        Let $X,Y\in\S(0,T)^d$.
        If Assumption~\ref{ass:1-non-lipschitz} is satisfied, then
        $f(\cdot ,X,Y)\in\M^2(0,T)$
        for all components $f=b_k,h_{kij},g_{ki}$, $1\leq i,j\leq n$, $1\leq k\leq d$.
    \end{lem}
    \begin{proof}
        Since $X,Y\in\S(0,T)$, there exist $N\in\Nbb$, $0=t_0<\ldots<t_N=T$ and $\xi_k,\eta_k\in\F{t_k}$ for $0\leq k\leq N-1$ such that
        \begin{equation*}
            X=\sum_{k=0}^{N-1}\xi_k\ind_{[t_k,t_{k+1})}, \qquad \qquad Y=\sum_{k=0}^{N-1}\eta_k\ind_{[t_k,t_{k+1})}.
        \end{equation*}
        Then we have
        \begin{equation*}
            f(\cdot,X,Y)=\sum_{k=0}^{N-1}f(\cdot,\xi_k,\eta_k)\ind_{[t_k,t_{k+1})}.
        \end{equation*}
        For $0\leq k\leq N-1$, define $f_k:=f(\cdot,\eta_k)\ind_{[t_k,t_{k+1})}:\,[0,T]\times \Rbb\times \Omega\rightarrow \Rbb$.
        By Assumption~\ref{ass:1-non-lipschitz}, $f_k(\cdot,x)=f(\cdot,x,\eta_k)\ind_{[t_k,t_{k+1})}\in\M^2(0,T)$ for all $x\in\Rbb$.
        The continuity assumption \eqref{ineq:ass-1-continuity} implies \eqref{ineq:A3-0} with $\rho(r):=\rho_1(r^2)$.
        Thus, we can apply Lemma~\ref{lem:A3} and obtain that $f_k(\cdot,\xi_k)\ind_{[t_k,T]}=f(\cdot,\xi_k,\eta_k)\ind_{[t_k,t_{k+1}]}\in\M^2(0,T)$ since $\xi_k\in\F{t_k}$ is $\CF_t$-measurable and bounded.
        Thus, $f(\cdot,X,Y)\in\M^2(0,T)$ as finite sum of elements in $\M^2(0,T)$.
    \end{proof}
    
    \begin{cor}\label{cor:integrable-coefficients}
        Let $X,Y\in\M^{2,d}(0,T)$ be such that
        \begin{equation}
            \sup_{0\leq s\leq T} \E\!\l[\l|X_s\r|^ 2\r] + \sup_{0\leq s\leq T}\E\!\l[\l|Y_s\r|^ 2\r] < \infty.
            \label{ineq:supE}
        \end{equation}
        If Assumption~\ref{ass:1-non-lipschitz} is satisfied, then
        $f(\cdot ,X,Y)\in\M^2(0,T)$
        for all components $f=b_k,h_{kij},g_{ki}$, $1\leq i,j\leq n$, $1\leq k\leq d$.
    \end{cor}
    \begin{proof}
        Since $X,Y\in\M^2(0,T)$ satisfy \eqref{ineq:supE}, there exist sequences $(X^m)_{m\in\Nbb},(Y^m)_{m\in\Nbb}$ in $\S(0,T)$ such that
        \begin{equation*}
            \lim_{m\rightarrow \infty} \sup_{0\leq s\leq T} \E\!\l[ \l|X^ m_s-X_s\r|^ 2\r]=0, \qquad \lim_{m\rightarrow \infty}\sup_{0\leq s\leq T}\E\!\l[ \l|Y^ m_s-Y_s\r|^ 2\r] =0
            .\numberthis\label{eq:lim-supE}
        \end{equation*}
        Lemma~\ref{lem:integrable-coefficients}  implies $f(\cdot,X^m,Y^m)\in\M^2(0,T)$ for all $m\in\Nbb$.
        By the continuity assumption \eqref{ineq:ass-1-continuity}, we have
        \begin{flalign*}
            \quad
            \int_0^T & \E\!\l[ \l|f(s,X^m_s,Y^ m_s)-f(s,X_s,Y_s)\r|^2 \r] \d s
            &&\\
            &\leq 
            \int_0^T \E\!\l[ \kappa(s)\rho_1\!\l(\l|X^m_s-X_s \r|^2\r) +K_s\,\rho_2\!\l(\l\|Y^m_s-Y_s\r\|_{\L^{2}}^2\r) \r] \d s
            \\&\leq 
            \int_0^T \kappa(s)\rho_1\!\l(\E\!\l[\l|X^m_s-X_s \r|^2\r]\r) +\E\!\l[K_s\r]\rho_2\!\l(\l\|Y^m_s-Y_s\r\|_{\L^{2}}^2\r) \d s
            \\&\leq 
             \int_0^T \kappa(s)\d s\,\rho_1\!\l( \sup_{0\leq w\leq T} \E\!\l[\l|X^m_w-X_w \r|^2\r]\r)
             +\int_0^T \E\!\l[ K_s\r] \d s\, \rho_2\!\l( \sup_{0\leq w\leq T} \E\!\l[ \l|Y^ m_w-Y_w\r|^ 2\r] \r)
             \\&\leq 
             C\,\rho_1\!\l( \sup_{0\leq w\leq T} \E\!\l[\l|X^m_w-X_w \r|^2\r]\r)
             +C\, \rho_2\!\l( \sup_{0\leq w\leq T} \E\!\l[ \l|Y^ m_w-Y_w\r|^ 2\r] \r)
             ,
        \end{flalign*}
        which tends to $0$ as $m\rightarrow\infty$ due to \eqref{eq:lim-supE}.        
        Thus, $f(\cdot,X,Y)\in\M^2(0,T)$ since $\M^2(0,T)$ is complete.
    \end{proof}
    
    For $\xi\in\L^{2,d}(t)$ and $X,Y\in\M^{2,d}(t,T)$, consider the process $\Phi^{t,\xi}(X,Y)$ given by
    \begin{align*}
        \Phi^{t,\xi}(X,Y)_s
        &:=\xi + \int_t^s b\!\l(u,X_u, Y_u\r) \d u + \sum_{i,j=1}^n \int_t^s h_{ij}\!\l(u,X_u, Y_u\r) \d \l<B^i,B^j\r>_u \\
        &\quad + \sum_{i=1}^n \int_t^s g_{i}\!\l(u,X_u, Y_u\r) \d B^i_u,
        \qquad t\leq s\leq T,
        \numberthis\label{eq:Phi}
    \end{align*}
    if it is well defined.

    \begin{lem}\label{lem:Phi-growth}
        Let $X,Y\in\M^{2,d}(t,T)$ be such that $\Phi(X,Y)\in\M^{1,d}(t,T)$.
        If Assumption~\ref{ass:1-non-lipschitz} is satisfied, then there exists a constant $\CK>0$ such that
        \begin{equation*}
            \E\!\l[ \sup_{t\leq w \leq s}\l\|\Phi^{t,\xi}(X,Y)\r\|^2 \r]
            \leq \CK \l( \l\|\xi\r\|_{\L^{2}}^2 + \E\!\l[ \int_t^s \kappa(u) \l|X_u\r|^2 + K_u\l(1 + \l\|Y_u\r\|_{\L^{2}}^2\r) \d u \r] \r)
        \end{equation*}
        for all $t\leq s\leq T$ and $\xi\in\L^{2,d}(t)$.
        Moreover, $\CK$ is independent of $X,Y$.
    \end{lem}
    \begin{proof}
        By Jensen's inequality, we have
        \begin{flalign*}
            \E\!\l[ \sup_{t\leq w\leq s}\l| \Phi^{t,\xi}(X,Y)_w\r|^2 \r]
            &\leq
            \l(2+n+n^2\r)\l( \E\!\l[ \l|\xi\r|^2\r] + \E\!\l[ \sup_{t\leq w\leq s}\l|\int_t^w b\!\l(u,X_u, Y_u\r) \d u \r|^2 \r] 
            \r.\\&\qquad \l.
            + \sum_{i,j=1}^n \E\!\l[ \sup_{t\leq w\leq s} \l| \int_t^w h_{ij}\!\l(u,X_u, Y_u\r) \d \l<B^i,B^j\r>_u\r|^2 \r] 
            \r. \\&\qquad \l.
            + \sum_{i=1}^n \E\!\l[ \sup_{t\leq w\leq s} \l| \int_t^w g_{i}\!\l(u,X_u, Y_u\r) \d B^i_u\r|^2 \r] \r)
            .\numberthis \label{ineq:sup}
        \end{flalign*}
        By the growth assumption \eqref{ineq:ass-1-growth}, we have for the integral with respect to $u$ that
        \begin{align*}
            \E\!\l[ \sup_{t\leq w\leq s}\l|\int_t^w b\!\l(u,X_u, Y_u\r) \d u \r|^2 \r]
            &\leq 
            \l(s-t\r) \E\!\l[\int_t^s\l| b\!\l(u,X_u, Y_u\r)\r|^2 \d u  \r]
            \\&\leq 
            \l(s-t\r) \E\!\l[ \int_t^s \kappa(u)\l|X_u\r|^2 + K_u\l(1 + \l\|Y_u\r\|_{\L^{2}}^2\r) \d u \r]
            .\numberthis \label{ineq:sup-b}
        \end{align*}
        Since $\Phi(X,Y)\in\M^{1}(t,T)$, we know that $h_{ij}(\cdot,X,Y)\in\M^1(t,T)$ and $g_i(\cdot,X,Y)\in\M^2(t,T)$ for $1\leq i,j\leq n$, and we can apply Lemmas~\ref{lem:Q-p-bound}~and~\ref{lem:I-p-bound}.
        Thus, the growth condition \eqref{ineq:ass-1-growth} yields for the integral with respect to $\l<B^i,B^j\r>$ with $1\leq i,j\leq n$ that
        \begin{flalign*}
            \E\!\l[ \sup_{t\leq w\leq s} \l| \int_t^w h_{ij}\!\l(u,X_u, Y_u\r) \d \l<B^i,B^j\r>_u\r|^2 \r]
            &\leq 
            \overline{\sigma}_{ij}^4\l(s-t\r) \E\!\l[ \int_t^s  \l| h_{ij}\!\l(u,X_u, Y_u\r)\r|^2  \d u \r]
            \\&\leq 
            \overline{\sigma}_{ij}^4\l(s-t\r) \E\!\l[ \int_t^s \kappa(u)\l|X_u\r|^2 + K_u\l(1 + \l\|Y_u\r\|_{\L^{2}}^2\r) \d u \r]
            ,\numberthis \label{ineq:sup-h}
        \end{flalign*}
        and for the integral with respect to $B^i$ with $1\leq i\leq n$ that
        \begin{flalign*}
            \E\!\l[ \sup_{t\leq w\leq s} \l| \int_t^w g_{i}\!\l(u,X_u, X_u\r) \d B^i_u\r|^2 \r]
            &\leq
            C_2\,\overline{\sigma}_{ii}^2
             \E\!\l[ \int_t^s \l| g_{i}\!\l(u,X_u, Y_u\r)\r|^2 \d u \r] 
            \\&\leq
            C_2\,\overline{\sigma}_{ii}^2 \E\!\l[ \int_t^s \kappa(u)\l|X_u\r|^2 + K_u\l(1 + \l\|Y_u\r\|_{\L^{2}}^2\r) \d u \r]
            .\numberthis\label{ineq:sup-g}
        \end{flalign*}
        Combining \eqref{ineq:sup}, \eqref{ineq:sup-b}, \eqref{ineq:sup-h} and \eqref{ineq:sup-g}, we obtain
        \begin{flalign*}
            \E\!\l[ \sup_{t\leq w\leq s}\l| \Phi^{t,\xi}(X,Y)_w\r|^2 \r] 
            &\leq 
            \CK \l( \l\|\xi\r\|_{\L^{2}}^2 + \E\!\l[ \int_t^s \kappa(u)\l|X_u\r|^2 + K_u\l(1 + \l\|Y_u\r\|_{\L^{2}}^2\r) \d u \r] \r)
            ,
        \end{flalign*}
        where the constant $\CK$ is given by
        \begin{equation*}
            \CK:=\l(2+n+n^2\r)\l( 1+ T + \sum_{i,j=1}^n \overline{\sigma}_{ij}^4 T + C_2 \sum_{i=1}^n\overline{\sigma}_{ii}^2\r)
            ,\numberthis\label{eq:K}
        \end{equation*}
        which is independent of $X,Y$.
    \end{proof}

    For $0\leq s\leq T$ and $d\geq 1$, let us introduce the space
    \begin{equation*}
        \H^{2,d}(t,T):=\l\{ X\in\M^{2,d}(s,T)\,:\, \E\!\l[\sup_{w\leq w\leq T}\l\|X_w\r\|^ 2\r] < \infty \r\}.
    \end{equation*}
    Clearly, $\H^2(0,T):=\H^{2,1}(0,T)\subseteq \M^2(0,T)$ is the completion of $\S(0,T)$ with respect to the norm
    \begin{equation*}
        \l\|\cdot\r\|_{\H^{2}}:\;\S(0,T)\rightarrow\Rbb_+,\qquad \l\|X \r\|_{\H^{2}}:=\E\!\l[\sup_{0\leq s\leq T}\l|X_s\r|^ 2\r]^{\frac{1}{2}}.
    \end{equation*}
    
    \begin{cor}
        Let $\xi\in\L^{2,d}(t)$ and $X,Y\in\H^{2,d}(t,T)$.
        If Assumption~\ref{ass:1-non-lipschitz} is satisfied, then $\Phi^{t,\xi}(X,Y)\in\M^{1,d}(t,T)$.
    \end{cor}
    \begin{proof}
        Corollary~\ref{cor:integrable-coefficients} implies that $Z:=f\!\l(\cdot,X\ind_{[t,T]},Y\ind_{[t,T]}\r)\in\M^2(0,T)$ for all components $f=b, h_{ij}, g_i$, $1\leq i,j\leq n$.
        In particular, $Z\ind_{[t,s]}=f\!\l(\cdot,X,Y\r)\ind_{[t,s]}\in\M^2(0,T)$ for all $t\leq s\leq T$. 
        Hence, $f(\cdot,X,Y)\in\M^2(t,T)$ and all integrals in \eqref{eq:Phi} are well defined.
        Thus, we have $\Phi^{t,\xi}(X,Y)\in\M^1(t,T)$ due to Lemmas~\ref{lem:I:M2-M2-map}, \ref{lem:Q:M1-M1-map} and \ref{lem:int:M1-M1-map}.
    \end{proof}
    
    \begin{cor}\label{cor:Phi-H2}
        Let $\xi\in\L^{2,d}(t)$ and $X,Y\in\H^{2,d}(t,T)$.
        If Assumption~\ref{ass:1-non-lipschitz} is satisfied, then $\Phi^{t,\xi}(X,Y)\in\H^{2,d}(t,T)$.
    \end{cor}
    \begin{proof}
        Lemma~\ref{lem:Phi-growth} implies that
        \begin{align*}
            \l\|\Phi^{t,\xi}(X,Y)\r\|_{\H^{2}}^2
            &=\E\!\l[ \sup_{t\leq s\leq T}  \l| \Phi^{t,\xi}(X,Y)_s\r|^2 \r]
            \\&\leq 
            \CK \l( \l\|\xi\r\|_{\L^{2}}^2 + \E\!\l[ \int_t^T \kappa(u)\l|X_u\r|^2 + K_u\l(1 + \l\|Y_u\r\|_{\L^{2}}^2\r) \d u \r] \r)
            \\&\leq 
            \CK \l( \l\|\xi\r\|_{\L^{2}}^2 +  \int_t^T \kappa(u) \E\!\l[ \l|X_u\r|^2\r] + \E\!\l[K_u\r]\l(1 + \l\|Y_u\r\|_{\L^{2}}^2\r) \d u \r)
            \\&\leq 
            \CK \l( \l\|\xi\r\|_{\L^{2}}^2 +  \int_t^T \kappa(u) \E\!\l[ \sup_{t\leq w\leq T}\l|X_w\r|^2\r] + \E\!\l[K_u\r]\l(1 + \E\l[ \sup_{t\leq w\leq T} \l|Y_w\r|^2\r]\r) \d u \r)
            \\&=
            \CK \l( \l\|\xi\r\|_{\L^{2}}^2 +  \int_t^T \kappa(u) \l\|X\r\|_{\H^{2}}^2 + \E\!\l[K_u\r]\l(1 + \l\|Y\r\|_{\H^{2}}^2\r) \d u \r)
            \\&\leq 
            \CK \l( \l\|\xi\r\|_{\L^{2}}^2 + C\l( 1 + \l\|X\r\|_{\H^{2}}^2 + \l\|Y\r\|_{\H^{2}}^2\r) \r)
            <\infty,
        \end{align*}
        i.e., $\Phi^{t,\xi}(X,Y)\in\H^{2}(t,T)$.
    \end{proof}
    
    \begin{lem}\label{lem:Phi-continuity}
        If Assumption~\ref{ass:1-non-lipschitz} is satisfied, then there exists a constant $\CK>0$ such that
        \begin{flalign*}
            \quad
            \E&\l[ \sup_{t\leq w\leq s}\l\|\Phi^{t,\xi}(X,Y)_w -\Phi^{t,\eta}(X',Y')_w\r\|^2\r]
            &&
            \\&\leq 
            \CK \l( \l\|\xi-\eta\r\|_{\L^{2}}^2 + \E\!\l[ \int_t^s \kappa\!\l(u\r) \rho_1\!\l( \l|X_u- X'_u \r|^2 \r) + K_u\,\rho_2\!\l(\l\|Y_u- Y'_u\r\|_{\L^{2}}^2 \r) \d u  \r] \r)
        \end{flalign*}
        for all $t\leq s\leq T$, $\xi,\eta\in\L^{2,d}(t)$, and $X,X',Y,Y'\in\H^{2,d}(t,T)$.
    \end{lem}
    \begin{proof}
        Jensen's inequality yields
        \begin{flalign*}
            \quad
            \E&\l[ \sup_{t\leq w\leq s}\l| \Phi^{t,\xi}(X,Y)_w - \Phi^{t,\eta}(X',Y')_w\r|^2 \r]
            &&
            \\&\leq 
            \l(2+n+n^2\r)\l( \E\!\l[\l|\xi-\eta\r|^2\r] + 
            \E\!\l[ \sup_{t\leq w\leq s}\l| \int_t^w b\!\l(u,X_u,Y_u\r) - b\!\l(u,X'_u,Y'_u\r) \d u \r|^2 \r]
            \r. \\ & \l. \qquad +
            \sum_{i,j=1}^n \E\!\l[ \sup_{t\leq w\leq s}\l| \int_t^w h_{ij}\!\l(u,X_u,Y_u\r) - h_{ij}\!\l(u,X'_u,Y'_u\r) \d \l<B^i,B^j\r>_u \r|^2\r]
            \r. \\ & \l. \qquad +
            \sum_{i=1}^n \E\!\l[ \sup_{t\leq w\leq s}\l| \int_t^w g_{i}\!\l(u,X_u,Y_u\r) - g_{i}\!\l(u,X'_u,Y'_u\r) \d B^i_u \r|^2\r]
            \r).
            \numberthis\label{ineq:Phi-continuity}
        \end{flalign*}
        Due to the continuity assumption \eqref{ineq:ass-1-continuity}, we have for the integral with respect to $u$ that
        \begin{flalign*}
            \quad
            \E&\l[ \sup_{t\leq w\leq s}\l| \int_t^w b\!\l(u,X_u,Y_u\r) - b\!\l(u,X'_u,Y'_u\r) \d u \r|^2 \r]
            &&
            \\&\leq 
            \l(s-t\r) \E\!\l[ \int_t^s \l| b\!\l(u,X_u,Y_u\r) - b\!\l(u,X'_u,Y'_u\r) \r|^2 \d u  \r]
            \\&\leq 
            \l(s-t\r) \E\!\l[ \int_t^s \kappa\!\l(u\r) \rho_1\!\l( \l|X_u- X'_u \r|^2 \r) + K_u\,\rho_2\!\l(\l\|Y_u- Y'_u\r\|_{\L^{2}}^2 \r) \d u  \r]
            .\numberthis\label{ineq:Phi-continuity-b}
        \end{flalign*}
        Analogously, by Lemmas~\ref{lem:Q-p-bound} and \ref{lem:I-p-bound}, we have for the integral with respect to $\l<B^i,B^j\r>$ with $1\leq i,j\leq n$ that
        \begin{flalign*}
            \quad
            \E&\l[ \sup_{t\leq w\leq s}\l| \int_t^w h_{ij}\!\l(u,X_u,Y_u\r) - h_{ij}\!\l(u,X'_u,Y'_u\r) \d \l<B^i,B^j\r>_u \r|^2\r]
            &&
            \\&\leq 
            \overline{\sigma}_{ij}^4\l(s-t\r)\E\!\l[\int_t^s \l| h_{ij}\!\l(u,X_u,Y_u\r) - h_{ij}\!\l(u,X'_u,Y'_u\r) \r|^2 \d u \r]
            \\&\leq 
            \overline{\sigma}_{ij}^4\l(s-t\r) \E\!\l[ \int_t^s \kappa\!\l(u\r) \rho_1\!\l( \l|X_u- X'_u \r|^2 \r) + K_u\,\rho_2\!\l(\l\|Y_u- Y'_u\r\|_{\L^{2}}^2 \r) \d u  \r]
            ,\numberthis\label{ineq:Phi-continuity-h}
        \end{flalign*}
        and for the integral with respect to $B^i$ with $1\leq i\leq n$ that
        \begin{flalign*}
            \quad
            \E&\l[ \sup_{t\leq w\leq s}\l| \int_t^w g_{i}\!\l(u,X_u,Y_u\r) - g_{i}\!\l(u,X'_u,Y'_u\r) \d B^i_u \r|^2\r]
            &&
            \\&\leq 
            C_2\,\overline{\sigma}_{ii}^2 \E\!\l[\int_t^s \l| g_{i}\!\l(u,X_u,Y_u\r) - g_{i}\!\l(u,X'_u,Y'_u\r) \r|^2 \d u \r]
            \\&\leq 
            C_2\,\overline{\sigma}_{ii}^2 \E\!\l[ \int_t^s \kappa\!\l(u\r) \rho_1\!\l( \l|X_u- X'_u \r|^2 \r) + K_u\,\rho_2\!\l(\l\|Y_u- Y'_u\r\|_{\L^{2}}^2 \r) \d u  \r]
            .\numberthis\label{ineq:Phi-continuity-g}
        \end{flalign*}
        Combining \eqref{ineq:Phi-continuity}, \eqref{ineq:Phi-continuity-b}, \eqref{ineq:Phi-continuity-h} and \eqref{ineq:Phi-continuity-g}, we obtain
        \begin{flalign*}
            \quad
            \E&\l[ \sup_{t\leq w\leq s}\l| \Phi^{t,\xi}(X,Y)_w - \Phi^{t,\eta}(X',Y')_w\r|^2 \r]
            &&
            \\&\leq 
            \CK \l(\l\|\xi-\eta\r\|_{\L^{2}}^2 +  \E\!\l[ \int_t^s \kappa\!\l(u\r) \rho_1\!\l( \l|X_u- X'_u \r|^2 \r) + K_u\,\rho_2\!\l(\l\|Y_u- Y'_u\r\|_{\L^{2}}^2 \r) \d u  \r] \r)
            ,
        \end{flalign*}
        where the constant $\CK$ can be chosen as in \eqref{eq:K}.
    \end{proof}
    
    \begin{cor}\label{cor:Phi-continuity}
        If Assumption~\ref{ass:1-non-lipschitz} is satisfied, then the map
        \begin{equation*}
            \L^{2,d}(t)\times\H^{2,d}(t,T)\times\H^{2,d}(t,T)\rightarrow \H^{2,d}(t,T),
            \qquad
            (\xi,X,Y)\mapsto \Phi^{t,\xi}(X,Y)
        \end{equation*}
        is continuous.
    \end{cor}
    \begin{proof}
        Let $\xi,\eta\in\L^{2}(t)$, and $X,X',Y,Y'\in\H^{2}(t,T)$.
        By Lemma~\ref{lem:Phi-continuity}, we have
        \begin{flalign*}
            \quad
            \| & \Phi^{t,\xi}(X,Y)-\Phi^{t,\eta}(X',Y')\|_{\H^{2}}^2
            &&
            \\&=
            \E\!\l[ \sup_{t\leq w\leq T}\l\|\Phi^{t,\xi}(X,Y)_w -\Phi^{t,\eta}(X',Y')_w\r\|^2\r]
            \\&\leq 
            \CK \l( \l\|\xi-\eta\r\|_{\L^{2}}^2 + \E\!\l[ \int_t^T \kappa\!\l(u\r) \rho_1\!\l( \l|X_u- X'_u \r|^2 \r) + K_u\,\rho_2\!\l(\l\|Y_u- Y'_u\r\|_{\L^{2}}^2 \r) \d u  \r] \r)
            \\&\leq 
            \CK \l( \l\|\xi-\eta\r\|_{\L^{2}}^2 + \int_t^T \kappa\!\l(u\r) \rho_1\!\l( \E\!\l[ \l|X_u- X'_u \r|^2 \r] \r) + \E\!\l[K_u\r] \rho_2\!\l(\E\!\l[ \l|Y_u- Y'_u \r|^2 \r] \r) \d u \r)
            \\&\leq
            \CK \l( \l\|\xi-\eta\r\|_{\L^{2}}^2 + \int_t^T \kappa\!\l(u\r)\rho_1\!\l( \l\|X- X'\r\|_{\H^{2}}^2 \r) \d u  + \int_t^T \E\!\l[K_u\r]\rho_2\!\l( \l\|Y- Y'\r\|_{\H^{2}}^2 \r) \d u \r). 
        \end{flalign*}
        Since $\rho_1$ and $\rho_2$ are continuous, the right-hand side vanishes, when $(\xi,X,Y)\rightarrow (\eta,X',Y')$ in $\L^{2}\times\H^{2}(t,T)\times\H^{2}(t,T).$
    \end{proof}

    \begin{prp}\label{prp:fixed-point}
        Let $\xi\in\L^{2,d}(t)$.
        If Assumption~\ref{ass:1-non-lipschitz} is satisfied, then the map
        \begin{equation*}
            \H^{2,d}(t,T) \rightarrow \H^{2,d}(t,T),\qquad X\mapsto \Phi^{t,\xi}(X):=\Phi^{t,\xi}(X,X)
        \end{equation*}
        has a unique fixed point.
    \end{prp}
    \begin{proof}        
        \emph{Existence:}
        Fix $\xi\in\L^{2}(t)$, and define the Picard sequence $\l(X^m\r)_{m\in\Nbb}$ recursively by
         \begin{equation*}
             X^0:=\xi\ind_{[t,T)}
             \qquad\text{and}\qquad
             X^{m+1}:=\Phi^{t,\xi}\!\l(X^m\r),\quad m\in\Nbb.
         \end{equation*}
         Since $\l\|X^0\r\|_{\H^{2}}=\l\|\xi\r\|_{\L^{2}}<\infty$ and thus $X^0\in\H^{2}(t,T)$, the sequence is well-defined in $\H^{2}(t,T)$ due to Corollary~\ref{cor:Phi-H2}.
         
         Let $\CK$ be the constant given in \eqref{eq:K}, and consider the function $q:\,[t,T]\rightarrow\Rbb$ defined by
         \begin{equation*}
             q(s):=\CK e^{\CK\int_t^s \gamma(u) \d u}\l(\l\|\xi\r\|_{\L^{2}}^2 +\int_t^s \E\!\l[K_u\r]\,e^{-\CK\int_t^u \gamma(w) \d w} \d u \r),
         \end{equation*}
         where $\gamma(u):=\E\l[K_u\r] + \kappa(u)\geq 0$ for $t\leq u\leq T$.
         Then $q$ is continuous, increasing and bounded, and the solution of the ODE
         \begin{align*}
             \d q(s) &= \CK\l(\E\!\l[ K_s\r] + \gamma(s)\,q(s) \r) \d s, \qquad t\leq s\leq T, \\
             q(t) &= \CK \l\|\xi\r\|_{\L^{2}}^2.
         \end{align*}
         By induction over $m\in\Nbb$, we show that
         \begin{equation*}
             \E\!\l[\sup_{t\leq w\leq s}\l|X^m_w\r|^2\r]\leq q(s)
         \end{equation*}
         for all $t\leq s\leq T$ and $m\in\Nbb$.
         Since $\CK>1$, we have for all $t\leq s\leq T$ that
         \begin{equation*}
             \E\!\l[\sup_{t\leq w\leq s}\l|X^0_w\r|^2\r]=\l\|\xi\r\|_{\L^{2}}^2\leq q(t)\leq q(s).
         \end{equation*}
         Suppose $\E\!\l[\sup_{t\leq w\leq s}\l|X^m_w\r|^2\r]\leq q(s)$ for all $t\leq s\leq T$ for some $m\in\Nbb$.
         By Lemma~\ref{lem:Phi-growth}, we have for all $t\leq s\leq T$ that
         \begin{align*}
             \E\!\l[\sup_{t\leq w\leq s}\l| X^{m+1}_w\r|^2\r]
             &=
             \E\!\l[\sup_{t\leq w\leq s}\l| \Phi^{t,\xi}(X^{m})_w\r|^2\r]
             \\&\leq 
             \CK \l( \l\|\xi\r\|_{\L^{2}}^2 + \int_t^s \E\!\l[ K_u\r] + \gamma(s) \E\!\l[\l|X^m_u\r|^2\r] \d u \r)
             \\&\leq 
             \CK \l( \l\|\xi\r\|_{\L^{2}}^2 + \int_t^s \E\!\l[ K_u\r] + \gamma(s) \E\!\l[\sup_{t\leq w\leq u}\l|X^m_w\r|^2\r] \d u \r)
             \\&\leq 
             \CK \l( \l\|\xi\r\|_{\L^{2}}^2 + \int_t^s \E\!\l[ K_u\r] + \gamma(s) q(u) \d u \r)
             \\&=
             q(s),
         \end{align*}
         which completes the induction.
         In particular, we established the uniform bound
         \begin{equation*}
             \sup_{m\in\Nbb} \E\!\l[\sup_{t\leq s\leq T}\l| X^{m}_s\r|^2\r] \leq q(T)=:\hat{q}.
         \end{equation*}

        Now, we show that the sequence $\l(X^m\r)_{m\in\Nbb}$ is Cauchy.
        Let $m,k\in\Nbb$ and $t\leq s\leq T$.
        By Lemma~\ref{lem:Phi-continuity}, we have
        \begin{flalign*}
            \quad
            \E&\l[\sup_{t\leq w\leq s}\l|X^{m+k+1}_w - X^{m+1}_w\r|^2 \r]
            &&\\
            &=
            \E\!\l[\sup_{t\leq w\leq s}\l|\Phi^{t,\xi}(X^{m+k})_w - \Phi^{t,\xi}(X^m)_w\r|^2 \r]
            \\&\leq
            \CK \int_t^s \gamma(u) \rho\!\l( \E\!\l[ \l|X^{m+k}_u-X^m_u\r|^2 \r] \r) \d u 
            \\&\leq 
            \CK \int_t^s \gamma(u) \rho\!\l( \E\!\l[ \sup_{t\leq w\leq u} \l|X^{m+k}_w-X^m_w\r|^2 \r] \r) \d u
            .\numberthis\label{ineq:Xmk-Xm-bound}       
         \end{flalign*}
    
        For $k,m\in\Nbb$, define the function $u_m^k:\,[t,T]\rightarrow\Rbb_+$ by
        \begin{equation*}
            u_m^k(s):=\E\!\left[ \sup_{t\leq w \leq s} \left|X^{m+k}_{w}-X^{m}_{w}\right|^2\right],
        \end{equation*}
        then plugging into \eqref{ineq:Xmk-Xm-bound} yields
        \begin{align*}
            u_{m+1}^k(s)
            &\leq
            \CK \int_t^{s} \gamma(w) \rho\!\l( u^k_m(w) \r) \d w
            .\numberthis\label{ineq:u-m-k}
        \end{align*}
        Taking the point-wise supremum of \eqref{ineq:u-m-k} over $k\in\Nbb$, we obtain
        \begin{align*}
            \sup_{k\in\Nbb} u_{m+1}^k(s)
            &\leq
            \sup_{k\in\Nbb} \CK \int_t^{s} \gamma(w) \rho\!\l( u^k_m(w) \r) \d w
            \\ &\leq
            \CK \int_t^{s} \gamma(w) \rho\!\l( \sup_{k\in\Nbb} u^k_m(w) \r) \d w
            .\numberthis\label{ineq:u-m}
        \end{align*}
        Define the point-wise limit superior
        \begin{equation*}
            u(s):= \limsup_{m\rightarrow \infty} \sup_{k\in\Nbb} u^k_m(s), \qquad t\leq s\leq T.
        \end{equation*}
        Since for all $k,m\in\Nbb$ and $t\leq s\leq T$, we have
        \begin{align*}
            0\leq u^k_m(s) 
            &=\sup_{t\leq w \leq s}\E\!\l[ \l|X^{m+k}_{w}-X^{m}_{w}\r|^2\r]
            \leq 
            \sup_{t\leq w \leq s} \l( \E\!\l[ \l|X^{m+k}_{w}\r|^2\r] + \E\!\l[ \l|X^{m}_{w}\r|^2\r] \r)
            \leq 2\,\hat{q},
        \end{align*}
        we also have $0\leq u(s)\leq 2\,\hat{q}$ for all $t\leq s\leq T$.
        In particular, $u$ and $u^k_m$ are Lebesgue integrable for any $k,m\in\Nbb$.
        Applying the Fatou-Lebesgue Theorem to \eqref{ineq:u-m}, yields
        \begin{equation*}
            0\leq u(s)\leq \CK \int_t^{s} \gamma(w) \rho\!\l( u(w) \r) \d w,
        \end{equation*}
        which implies $u\equiv0$ due to Bihari's inequality.
        Hence, $(X^m)_{m\in\mathbb{N}}$ is a Cauchy sequence with respect to the norm $\|\cdot\|_{\H^{2}}$.
        Since $\H^{2}(t,T)$ is complete, we have
        \begin{equation*}
            \lim_{m\rightarrow\infty}
            X^m=:X\in\H^{2}(t,T).
        \end{equation*}
        By Corollary~\ref{cor:Phi-continuity}, the map $X\mapsto \Phi^{t,\xi}(X)$ is continuous.
        Thus, we have
        \begin{align*}
            X=\lim_{m\rightarrow\infty}X^{m+1}
            =\lim_{m\rightarrow\infty}\Phi^{t,\xi}\!\l(X^m\r)
            =\Phi^{t,\xi}\!\l(\lim_{m\rightarrow\infty}X^m\r)
            =\Phi^{t,\xi}(X)
        \end{align*}
        in $\H^{2}(t,T)$.
        That is, $X$ is a fixed point of $\Phi^{t,\xi}$.

        \emph{Uniqueness.}
        Suppose $X,Y\in\H^{2}(t,T)$ satisfy $X=\Phi^{t,\xi}(X)$ and $Y=\Phi^{t,\xi}(Y)$.
        Lemma~\ref{lem:Phi-continuity} yields for all $t\leq s\leq T$ that
        \begin{flalign*}
            \quad
            \E&\l[\sup_{t\leq w\leq s}\l|X_w-Y_w\r|^2 \r]
            &&
            \\&=
            \E\!\l[\sup_{t\leq w\leq s}\l|\Phi^{t,\xi}(X)_w-\Phi^{t,\xi}(Y)_w\r|^2 \r]
            \\&\leq 
            \CK \int_t^s \gamma(u) \rho\!\l( \E\!\l[ \sup_{t\leq w\leq u}\l|X_w-Y_w\r|^2 \r] \r) \d u
            .\numberthis\label{ineq:X-Y}
        \end{flalign*}
        Define the function $u:\,[t,T]\rightarrow\Rbb_+$ by
        \begin{equation*}
            u(s):=\E\!\l[ \sup_{t\leq w\leq s}\l|X_w-Y_w\r|^2 \r].
        \end{equation*}
        Plugging $u$ into \eqref{ineq:X-Y}, we obtain
        \begin{equation*}
            u(s)\leq \CK \int_t^{s} \gamma(w) \rho\!\l( u(w) \r) \d w,
        \end{equation*}
        and Bihari's inequality yields $u\equiv 0$.
        Thus, $\Phi^{t,\xi}$ has a unique fixed point.
    \end{proof}

    Immediately, we deduce the following result.
    \begin{cor}\label{cor:fixed-point}
        Let $\xi\in\L^{2,d}(t)$ and $Y\in\H^{2,d}(t,T)$.
        If Assumption~\ref{ass:1-non-lipschitz} is satisfied, then the map
        \begin{equation*}
            \H^{2,d}(t,T) \rightarrow \H^{2,d}(t,T),\qquad X\mapsto \Phi^{t,\xi}(X,Y)
        \end{equation*}
        has a unique fixed point.
    \end{cor}

    \begin{thm}\label{thm:existence-uniqueness-ass-1}
        Let $\xi\in\L^{2,d}(t)$.
        If Assumption~\ref{ass:1-non-lipschitz} is satisfied, then there exists a unique $X\in\H^{2,d}(t,T)$ which solves \eqref{eq:MF-GSDE} with $X_t=\xi$.
        
        If additionally 
        \begin{equation*}
            \gamma(s)=\E\l[K_s\r] + \kappa(s),\qquad t\leq s\leq T,
        \end{equation*}
        is bounded, then the solution $X$ is unique in $\M^{2,d}(t,T)$ and
        \begin{equation*}
            \E\!\l[\sup_{t\leq s\leq T}\l\|X_s\r\|^2 \r]<\infty.
        \end{equation*}
    \end{thm}
    \begin{proof}
        From the definition of $\Phi^{t,\xi}(X,Y)$ in \eqref{eq:Phi}, we deduce that $X\in\M^2(t,T)$ is a solution of \eqref{eq:MF-GSDE} with $X_t=\xi$ if, and only if, $X=\Phi^{t,\xi}(X,X)$.

        Suppose $X\in\M^2(t,T)$ satisfies $X=\Phi^{t,\xi}(X,X)$, then $\Phi^{t,\xi}(X,X)=X\in\M^2(t,T)\subseteq \M^1(t,T)$.
        Set $M:=\sup_{t\leq s\leq T}\gamma(s)$.
        Lemma~\ref{lem:Phi-growth} yields that
        \begin{align*}
            \E\!\l[ \sup_{t\leq s\leq T}\l|X_s\r|^2 \r] 
            &= 
            \E\!\l[ \sup_{t\leq s\leq T}\l|\Phi^{t,\xi}(X,X)_s\r|^2 \r]
            \\&\leq 
            \CK \l( \l\|\xi\r\|_{\L^2}^2 + \l\|K\r\|_{\M^1} + \int_t^T \gamma(u) \E\l[ \l|X_u\r|^2 \r] \d u \r)
            \\&\leq 
            \CK \l( \l\|\xi\r\|_{\L^2}^2 + \l\|K\r\|_{\M^1} + M \l\|X\r\|_{\M^2}^2 \r) < \infty.
        \end{align*}
        That is, $X\in\H^2(t,T)$.
        By Proposition~\ref{prp:fixed-point}, there exists a unique $X\in\H^2(t,T)$ with $X=\Phi^{t,\xi}(X)=\Phi^{t,\xi}(X,X)$, which implies the desired result.
    \end{proof}
    
    \section{Comparison to existing Literature}\label{sec:discussion}
    In this section, we will discuss how our results relate to the existing literature on mean-field $G$-SDEs. 
    In particular, we show that Theorem~\ref{thm:existence-uniqueness-ass-1} generalises the existing results in the literature.

    First, let us show that the $G$-SDE considered in \cite{sun_distribution_2023} is a special case of the mean-field $G$-SDE \eqref{eq:MF-GSDE}.
    For the sake of convenience, we recall the definitions from Section 3 in \cite{sun_distribution_2023}.
    Let $\CD$ denote the space of all functionals $F:\,\textnormal{Lip}(\Rbb^d)\rightarrow \Rbb$ which are monotonous, positive homogeneous, sub-additive, constant-invariant and such that
    \begin{equation*}
        \sup_{L_\phi \leq 1} \l|F_1(\phi) - \phi(0)\r| <\infty,
    \end{equation*}
    where the supremum ranges over all Lipschitz continuous $\phi:\,\Rbb^d\rightarrow \Rbb$ with Lipschitz constant $L_\phi\leq 1$.
    Moreover, define the metric $d_1$ on $\CD$ by
    \begin{equation*}
        d_1(F_1,F_2):=\sup_{L_\phi \leq 1} \l|F_1(\phi) - F_2(\phi)\r|, \qquad f,g\in\CD.
    \end{equation*}
    For any $\xi\in\L^{1,d}$, define the functional
    \begin{equation*}
        F_\xi:\; \textnormal{Lip}(\Rbb^d) \rightarrow\Rbb,\qquad\phi \mapsto \E\!\l[ \phi(\xi)\r].
    \end{equation*}
    Clearly, $F_\xi\in\CD$, cf. also Remark 3.2 in \cite{sun_distribution_2023}.
    Moreover, we have for any $\xi,\eta\in\L^{1,d}$ that
    \begin{align*}
        d_1(F_\xi,F_\eta) 
        &= 
        \sup_{L_\phi \leq 1} \l|F_\xi(\phi) - F_\eta(\phi)\r|
        \\&=
        \sup_{L_\phi \leq 1} \l| \E\!\l[ \phi(\xi)\r] - \E\!\l[ \phi(\eta)\r] \r|
        \\&\leq 
        \sup_{L_\phi \leq 1}  \E\!\l[ \l| \phi(\xi) - \phi(\eta) \r| \r] 
        \\&\leq 
        \sup_{L_\phi \leq 1}  L_\phi\, \E\!\l[ \l| \xi - \eta \r| \r] 
        \\&=
        \l\|\xi-\eta\r\|_{\L^{1}}
        .\numberthis\label{ineq:d1-L1-bound}
    \end{align*}
    
    In \cite{sun_distribution_2023}, the authors consider dynamics of the form
    \begin{align*}
        \d X_s 
        &= 
        b\!\l(s,X_s,F_{X_s}\r) \d s 
        + h\!\l(s,X_s,F_{X_s}\r) \d \l<B\r>_s 
        + g\!\l(s,X_s,F_{X_s}\r) \d B_s
        .\numberthis\label{eq:SDE-sun-1}
    \end{align*}
    The coefficients $\Tilde{b},\Tilde{h},\Tilde{g}$ are defined on $[0,T]\times\Rbb^d\times\CD$ such that the components $\tilde f=\Tilde{b}_k,\tilde h_{kij}, \Tilde{g}_{ki}$ with $1\leq i,j\leq n$, $1\leq k\leq d$ satisfy
    \begin{equation*}
        | \tilde f(t,x,F_1)- \tilde f(t,y,F_2) | \leq K\l(\l\|x-y\r\| + d_1(F_1,F_2)\r)
    \end{equation*}
    for some constant $K>0$, cf. (H1) in \cite{sun_distribution_2023}.
    
    Now, let us define $b,g,h$ on $[0,T]\times\Rbb^d\times \L^{2,d}\times \Omega $ by
    \begin{align*}
        b(s,x,\xi,\omega)&:=\tilde{b}(s,x,F_{\xi}) \\
        h(s,x,\xi,\omega)&:=\tilde{h}(s,x,F_{\xi}) \\
        g(s,x,\xi,\omega)&:=\tilde{g}(s,x,F_{\xi})
    \end{align*}
    for all $0\leq s\leq T$, $x\in\Rbb^d$, $\xi\in\L^{2,d}$, $\omega\in\Omega$.
    
    Note that the coefficients $b,h,g$ are deterministic. 
    For the components $f=b_k,h_{kij},g_{ki}$, $1\leq i,j\leq n$, $1\leq k\leq d$, we have that
    \begin{align*}
        \l| f(t,x,\xi,\omega)-f(t,y,\eta,\omega) \r|
        &= 
        | \tilde f(t,x,F_\xi)-\tilde f(t,y,F_\eta) |
        \\ &\leq 
        K\l(\l\|x-y\r\| + d_1(F_\xi,F_\eta)\r)
        \\ &\leq 
        K\l(\l\|x-y\r\| + \E\!\l[ \l| \xi - \eta \r| \r] \r)
    \end{align*}
    for all $\omega\in\Omega$, $0\leq s\leq T$, $x,y\in\Rbb^d$ and $\xi,\eta\in\L^{2,d}$
    due to \eqref{ineq:d1-L1-bound}.
    In particular, Jensen's inequality yields 
    \begin{align*}
        \l| f(t,x,\xi,\omega)-f(t,y,\eta,\omega) \r|^2 
        &\leq K^2 \l(\l\|x-y\r\| + \E\!\l[ \l| \xi - \eta \r| \r] \r)^2
        \\&\leq 2\,K^2 \l(\l\|x-y\r\|^2 + \E\!\l[ \l| \xi - \eta \r|^2 \r] \r)
        \\&= 2\,K^2 \l(\l\|x-y\r\|^2 + \l\|\xi-\eta\r\|_{\L^{2}}^2 \r)
        ,
    \end{align*}
    i.e., the coefficients $b,h,g$ satisfy \eqref{ineq:ass-1-continuity} from Assumption \ref{ass:1-non-lipschitz}. 
    Thus, Theorem~\ref{thm:existence-uniqueness-ass-1} implies Theorem~4.1 in \cite{sun_distribution_2023}.

    In \cite{sun_mean-field_2020}, the author considers the $G$-SDE
    \begin{align*}
        \d X_s & =\E\!\l[\Tilde{b}\l(s,x,X_s\r)\r]\bigg|_{x=X_s} \d s + \sum_{i,j=1}^n \E\!\l[\Tilde{h}_{ij}\!\l(s,x,X_s\r)\r]\bigg|_{x=X_s} \d\l<B^i,B^j\r>_s \\
        &\quad + \sum_{i=1}^n\E\Big[\Tilde{g}_{i}\!\l(s,x,X_s\r)\Big]\bigg|_{x=X_s} \d B^i_s,
    \end{align*}
    where the functions $\Tilde{b},\Tilde{h}_{ij},\Tilde{g}_i$, $1\leq i,j\leq n$ are defined on $[0,T]\times\Rbb\times\Rbb$.
    That is,
    \begin{equation*}
        \E\l[ \Tilde{f}(s,x,X_s) \r]\Big|_{x=X_s} = F_{X_s} \Tilde{f}(s,x,\cdot)\Big|_{x=X_s}=:f(s,X_s,F_{X_s})
    \end{equation*}
    for $\Tilde{f}=\Tilde{b},\Tilde{h}_{ij},\Tilde{g}_i$, $1\leq i,j\leq n$.
    Thus, it is a special case of \eqref{eq:SDE-sun-1} and thus also of \eqref{eq:MF-GSDE} with $d=1$.
    
    \appendix

    \section{Auxiliary Results}\label{app:auxiliary_results}
    
    \begin{lem}\label{lem:A1}
        Let $0\leq t\leq T$, $\eta\in\F{t}$, and $X\in\S(0,T)$. 
        Then $X\eta\ind_{[t,T]}\in\S(0,T)$.
    \end{lem}
    \begin{proof}
        Since $X\in\S(0,T)$, it is of the form
        \begin{equation*}
            X=\sum_{k=0}^{N-1}\xi_k\ind_{[t_k,t_{k+1})}
        \end{equation*}
        for some $N\in\Nbb$, $0=t_0<\ldots<t_N=T$ and $\xi_k\in\F{t_k}$, $0\leq k\leq N-1$.
        Fix $0\leq l\leq N$ such that $t_l \leq  t < t_{l+1}$.
        For each $l \leq  k\leq N-1$, we have $\F{t}\subseteq\F{t_k}$ and $\F{t_k}$ is closed under multiplication, i.e., $\xi_k\eta\in\F{t_k}$.
        Hence,
        \begin{equation*}
            X\eta \ind_{[t,T]} =\xi_l\eta\ind_{[t,t_{l+1})} + \sum_{k=l+1}^{N-1}\xi_k\eta \ind_{[t_k,t_{k+1})} \in\S(0,T),
        \end{equation*}
        as desired.
    \end{proof}
    
    \begin{cor}\label{cor:A2}
        Let $0\leq t\leq T$, $\eta\in\F{t}$, and $X\in\M^p(0,T)$ $p\geq 1$. 
        Then $X\eta \ind_{[t,T]}\in\M^p(0,T)$.
    \end{cor}
    \begin{proof}
        Since $X\in\M^1(0,T)$ there exists a sequence $(X^m)_{m\in\Nbb}$ in $\S(0,T)$ such that 
        \begin{equation*}
            \lim_{m\rightarrow \infty} \l\|X^m-X\r\|_{\M^p}=0.
        \end{equation*}
        By Lemma~\ref{lem:A1}, we have $X^m\eta\ind_{[t,T]}\in\S(0,T)$ for all $m\in\Nbb$.
        Since $\eta\in\F{t}$, there exists a constant $M>0$ such that $\l|\eta\r|<M$.
        Thus,
        \begin{equation*}
            \lim_{m\rightarrow \infty} \l\|X^m\eta \ind_{[t,T]}-X\eta\ind_{[t,T]} \r\|_{\M^p}\leq \lim_{m\rightarrow \infty} M \l\|X^m-X\r\|_{\M^p}= 0.
        \end{equation*}
        Since $\M^p(0,T)$ is the completion of $\S(0,T)$ with respect to the $\M^p$-norm, we obtain that $X\eta\ind_{[t,T]}\in\M^p(0,T)$.
    \end{proof}
    \begin{cor}\label{cor:A3}
        Let $X\in\S(0,T)$ and $Y\in\M^p(0,T)$. Then $XY\in\M^p(0,T)$.
    \end{cor}
    \begin{proof}
        Since $X\in\S(0,T)$, it is of the form
        \begin{equation*}
            X=\sum_{k=0}^{N-1}\xi_k\ind_{[t_k,t_{k+1})}
        \end{equation*}
        for some $N\in\Nbb$, $0=t_0<\ldots<t_N=T$ and $\xi_k\in\F{t_k}$, $0\leq k\leq N-1$.
        We have
        \begin{equation*}
            XY=\sum_{k=0}^{N-1}\xi_kY\ind_{[t_k,t_{k+1})}.
        \end{equation*}
        Corollary~\ref{cor:A2} implies $\xi_kY\ind_{[t_k,t_{k+1})}\in\M^p(0,T)$, and thus $XY\in\M^p(0,T)$ as finite sum of elements in $\M^p(0,T)$.
    \end{proof}
    
    \begin{lem}\label{lem:A3}
        Let $(E,d_E)$ be a metric space, and $\xi:\,\Omega\rightarrow E$ be $\CF_t$-measurable and such that $\textnormal{image } \xi \subseteq K$ for some compact $K\subseteq E$. 
        Moreover, let $p\geq 1$ and $f:\,[0,T]\times E\times\Omega \rightarrow \Rbb$ be such that $f(\cdot,x)\in\M^p(0,T)$ for all $x\in E$.
        If there exist a constant $\delta>0$, a process $M\in\M^1(t,T)$ and an increasing function $\rho:\,\Rbb_+\rightarrow\Rbb_+$ with $\lim_{\epsilon\downarrow 0}\rho(\epsilon)=0$ such that
        \begin{equation*}
            \l| f(s,x,\omega) - f(s,y,\omega) \r|^p \leq M_s\!\l(\omega\r)\rho(d_E(x,y))
            \numberthis\label{ineq:A3-0}
        \end{equation*}
        for all $x,y\in E$ with $d_E(x,y)<\delta$, quasi all $\omega\in\Omega$, and almost all $0\leq s\leq T$;
        then $f(\cdot,\eta)\ind_{[t,T]}\in\M^p(0,T)$.
    \end{lem}
    \begin{proof}
        Let $0<\epsilon<\delta$, and let $B_\epsilon(y)$ denote the open $\epsilon$-ball centered at $y\in E$.
        Then $\{B_\epsilon(y)\}_{y\in K}$ is an open cover of $K$.
        Since $K$ is compact, we can fix a finite $I\subseteq K$ such that $\{B_\epsilon(y)\}_{y\in I}$ is an open cover of $K$.
        Moreover, there exists a partition of unity $\{\psi^\epsilon_y\}_{y\in I}$ subordinate to $\{B_\epsilon(y)\}_{y\in I}$.
        That is, $\psi^\epsilon_y:\,E\rightarrow [0,1]$ is continuous with $\textnormal{supp }\psi^\epsilon_y \subseteq B_\epsilon(y)$ for each $y\in I$, and $\sum_{y\in I}\psi^\epsilon_y(x)=1$ for all $x\in K$.

        Define the function $f^\epsilon_\eta:\,[0,T]\times \Omega\rightarrow\Rbb$ by
        \begin{equation*}\label{eq:A3-1}
            f^\epsilon_\eta
            :=\sum_{y\in I}f\!\l(\cdot,y\r)\psi^\epsilon_y\!\l(\eta\r)\ind_{[t,T]}
            .
        \end{equation*}
        For each $y\in I$, the concatenation $\psi^\epsilon_y(\eta):\,\Omega\rightarrow [0,1]$ is $\CF_t$-measurable and bounded, i.e., $\psi^\epsilon_y(\eta)\in\F{t}$.
        Further, $f(\cdot,y)\in\M^p(0,T)$ implies $f\!\l(\cdot,y\r)\psi^\epsilon_y\!\l(\eta\r)\ind_{[t,T]}\in\M^p(0,T)$ due to Corollary~\ref{cor:A2}.
        Thus, $f^\epsilon_\eta\in\M^p(0,T)$ as finite sum of elements in $\M^p(0,T)$.
        
        For $t\leq s\leq T$, we have
        \begin{align*}
            \E\!\l[\l| f(s,\eta) - f^\epsilon_\eta(s) \r|^p \r]
            &=
            \E\!\l[\l| f(s,\eta) - \sum_{y\in I}f(s,y)\,\psi^\epsilon_y(\eta) \r|^p  \r]
            \\&=
            \E\!\l[\l| \sum_{y\in I} \l(f(s,\eta) - f(s,y)\r)\psi^\epsilon_y(\eta) \r|^p \r]
            \\&\leq 
            \E\!\l[ \sum_{y\in I} \l| f(s,\eta) - f(s,y)\r|^p \psi^\epsilon_y(\eta) \r]
            .\numberthis \label{ineq:A3-2}
        \end{align*}
        There exists a polar set $N\subseteq \Omega$ such that \eqref{ineq:A3-0} holds for all $\omega\in\Omega\setminus N$, almost every $0\leq s\leq T$ and all $x,y\in E$ with $d_E(x,y)<\delta$. 
        By construction, $\psi^\epsilon_y(\eta)=0$ on $\{\eta\notin B_\epsilon(y)\}$.
        Further,
        \begin{equation*}
            \l| f(s,\eta) - f(s,y)\r|^p
            \leq 
            M_s\,\rho\!\l(d_E(\eta,y)\r)
            <
            M_s\, \rho\!\l(\epsilon\r)\qquad\text{on }\l\{\eta\in B_\epsilon(y)\r\}\setminus N
        \end{equation*} 
        for almost every $t\leq s\leq T$.
        Hence,
        \begin{equation*}
            \l| f(s,\eta) - f(s,y)\r|^p \psi^\epsilon_y(\eta) < M_s\,\rho\!\l(\epsilon\r) \psi^\epsilon_y(\eta) \qquad \text{on }\Omega\setminus N
            \label{ineq:A3-3}
        \end{equation*}
        for almost every $t\leq s\leq T$.
        Since $N$ is polar, we deduce from \eqref{ineq:A3-2} that
        \begin{align*}
            \l\| f(s,\eta)\ind_{[t,T]} - f^\epsilon_\eta(s) \r\|_{\M^p}^p
            &=
            \int_t^T \E\!\l[\l| f(s,\eta) - f^\epsilon_\eta(s) \r|^p \r] \d s 
            \\&\leq 
            \int_t^T \E\!\l[\sum_{y\in I} \l| f(s,\eta) - f(s,y)\r|^p \psi^\epsilon_y(\eta)  \r] \d s
            \\&< 
            \int_t^T  \E\!\l[\sum_{y\in I} M_s\,\rho\!\l(\epsilon\r) \psi^\epsilon_y(\eta) \r] \d s 
            \\&=
            \int_t^T \E\!\l[ M_s \r] \d s\; \rho\!\l(\epsilon\r)
            \\&\leq 
            \l\|M\r\|_{\M^1}\rho\!\l(\epsilon\r)
            .
        \end{align*}
        Since $\lim_{\epsilon\downarrow 0}\rho(\epsilon)=0$, we have $f^\epsilon_\eta\rightarrow f(\cdot,\eta)$ with respect to the $\M^p$-norm as $\epsilon\downarrow 0$.
        Finally, we obtain $f(\cdot,\eta)\ind_{[t,T]}\in\M^p(0,T)$ because $\M^p(0,T)$ is complete.
    \end{proof}
    
    \section{Proofs to Results in Section~\ref{sec:G-setting}}\label{app:proofs_gsetting}
    
    \begin{proof}[Proof of Lemma~\ref{lem:Lp-sum-Mp}]
        For each $0\leq k\leq N-1$, there exists a sequence $(\xi_k^m)_{m\in\Nbb}$ in $\F{t_k}$ with
        \begin{equation*}
            \lim_{m\rightarrow\infty}\E\l[ \l|\xi_k-\xi_k^m\r|^p\r] =0.
        \end{equation*}
        For each $m\in\Nbb$, we have
        \begin{equation*}
            X^m:=\sum_{k=0}^{N-1} \xi_k^m\ind_{\l[t_k,t_{k+1}\r)}\in\S(t,T).
        \end{equation*}
        Further,
        \begin{align*}
            \lim_{m\rightarrow\infty} \int_t^T\E\l[ \l| X_s - X_s^m \r|^p \r] \d s
            &=
            \lim_{m\rightarrow\infty}\sum_{k=0}^{N-1} \E\l[  \l| \xi_k - \xi_k^m \r|^p  \r]\l( t_{k+1}-t_k\r)
            =0.
        \end{align*}
        Hence, $X^m\in\S(t,T)$ converges to $X$ with respect to $\l\|\cdot\r\|_{\M^p}$ as $m\rightarrow \infty$. 
        Thus $X\in\M^p(t,T)$ since $\M^p(t,T)$ the completion of $\S(t,T)$.
    \end{proof}

    \begin{proof}[Proof of Lemma~\ref{lem:B-M2}]
        For $m\in\Nbb$, set $t^m_k:=\frac{k\,T}{m}$ and define
        \begin{equation*}
            B^{a,m}:=\sum_{k=1}^{m-1} B^a_{t^m_k} \ind_{[t^m_k,t^m_{k+1})}.
        \end{equation*}
        Clearly, $B^a_{t^m_k}\in\L^2(t^m_k)$ for $0\leq k\leq m-1$ and hence $B^{a,m}\in\M^2(0,T)$ due to Lemma~\ref{lem:Lp-sum-Mp}.
        We have
        \begin{align*}
            \int_0^T \E\l[ \l|B^a_s -B^{a,m}_s\r|^2\r] \d s 
            &= 
            \sum_{k=0}^{m-1} \int_{t^m_k}^{t^m_{k+1}} \E\l[ \l|B^a_s -B^{a}_{t^m_k}\r|^2\r] \d s
            \\&=
            \sum_{k=0}^{m-1} \int_{t^m_k}^{t^m_{k+1}} \overline{\sigma}_{aa}^2 \l(s-t^m_k\r) \d s
            \\&\leq 
            \sum_{k=0}^{m-1} \overline{\sigma}_{aa}^2 \frac{T^2}{m^2}
            \\&= \frac{\overline{\sigma}_{aa}^2\,T^2}{m},
        \end{align*}
        which tends to $0$ when $m\rightarrow \infty$, i.e., $B^{a,m}\rightarrow B^a$ with respect to $\|\cdot\|_{\M^2}$.
        Since $\M^2(0,T)$ is complete, we deduce $B^a\in\M^2(0,T)$.
    \end{proof}

    \begin{proof}[Proof of Lemma~\ref{lem:I:M2-M2-map}]
        First, suppose $X\in\S(0,T)$. 
        Then there exist $m\in\Nbb$, $0=t_0<\ldots<t_m=T$, and $\xi_k\in\F{t_k}$, $0\leq k\leq N-1$ such that
        \begin{equation*}
            X=\sum_{k=0}^{m-1} \xi_k\ind_{\l[t_k,t_{k+1}\r)}.
        \end{equation*}
        We have
        \begin{align*}
            Z 
            &= 
            \int_0^\cdot X_u \d B^a_u 
            \\&=
            \sum_{k=0}^{m-1} \xi_k \l(B^a_{t_{k+1}\wedge \cdot} -B^a_{t_k\wedge \cdot}\r)
            \\&=
            \sum_{k=0}^{m-1} \xi_k \l(B^a_\cdot - B^a_{t_k}\r)\ind_{[t_k,t_{k+1})} + \sum_{k=0}^{m-1} \xi_k\l(B^a_{t_{k+1}} -B^a_{t_k}\r)\ind_{[t_{k+1},T)}
            \numberthis\label{eq:M2-M2-1}
        \end{align*}
        For the former sum on the right-hand side of \eqref{eq:M2-M2-1}, note that
        \begin{equation*}
            \xi_k \l(B^a_\cdot - B^a_{t_k}\r)\ind_{[t_k,t_{k+1})} = \xi_k\ind_{[t_k,t_{k+1})}\cdot \l(B^a_\cdot - B^a_{t_k}\ind_{[t_k,t_{k+1})}\r).
        \end{equation*}
        Clearly, $B^a_{t_k}\in\L^2(t_{k})$ and thus $B^a_{t_k}\ind_{[t_k,t_{k+1})}\in\M^2(0,T)$ due to Lemma~\ref{lem:Lp-sum-Mp}.
        By Lemma~\ref{lem:B-M2}, we have $B^a\in\M^2(0,T)$ and thus $(B^a-B^a_{t_k}\ind_{[t_k,t_{k+1})})\in\M^2(0,T)$.
        Since $\xi_k\ind_{[t_{k+1},T)}\in\S(0,T)$, we obtain $\xi_k \l(B^a_\cdot - B^a_{t_k}\r)\ind_{[t_k,t_{k+1})}\in\M^2(0,T)$ due to Corollary~\ref{cor:A3}.

        Similarly, for the latter sum on the right-hand side of \eqref{eq:M2-M2-1},
        \begin{equation*}
            \xi_k\l(B^a_{t_{k+1}} -B^a_{t_k}\r)\ind_{[t_{k+1},T)}
            =\xi_k\ind_{[t_{k+1},T)}\cdot \l(B^a_{t_{k+1}} -B^a_{t_k}\r)\ind_{[t_{k+1},T)}.
        \end{equation*}
        We have $B^a_{t_k},B^a_{t_{k+1}}\in\L^2(t_{k+1})$ and thus $(B^a_{t_{k+1}} -B^a_{t_k})\ind_{[t_{k+1},T)} \in\M^2(0,T)$ due to Lemma~\ref{lem:Lp-sum-Mp}. 
        Corollary~\ref{cor:A3} yields $\xi_k(B^a_{t_{k+1}} -B^a_{t_k})\ind_{[t_{k+1},T)}\in\M^2(0,T)$.

        Finally, we deduce that $Z\in\M^2(0,T)$ as finite sum of elements in $\M^2(0,T)$.

        Now, suppose $X\in\M^2(0,T)$.
        Then there exists a sequence $(X^m)_{m\in\Nbb}$ in $\S(0,T)$ with
        \begin{equation*}
            \lim_{m\rightarrow\infty}\l\|X-X^m\r\|_{\M^2} =0.
        \end{equation*}
        For each $m\in\Nbb$, define $Z^m$ by
        \begin{equation*}
            Z^m_t:=\int_0^t X^m_s \d B^a_s=\CI_a\l(X^m\ind_{[0,t)}\r) ,\qquad 0\leq t\leq T.
        \end{equation*}
        Then $Z^m\in\M^2(0,T)$.
        By Lemma~\ref{lem:I-properties}, we have
        \begin{equation*}
            \E\l[ \l| Z^m_t - Z_t \r|^2\r]
            \leq \overline{\sigma}_{aa}^2 \l\| \l(X^m - X\r)\ind_{[0,t)} \r\|_{\M^2}^2 
            \leq \overline{\sigma}_{aa}^2 \l\| X^m - X \r\|_{\M^2}^2,
        \end{equation*}
        and thus
        \begin{equation*}
            \lim_{m\rightarrow\infty} \int_0^T \E\l[ \l| Z^m_t - Z_t \r|^2\r] \d t 
            \leq \overline{\sigma}_{aa}^2\,T \l\| X^m - X \r\|_{\M^2}^2
            = 0.
        \end{equation*}
        Hence, $Z\in\M^2(0,T)$ since $\M^2(0,T)$ is complete.
    \end{proof}

    \begin{proof}[Proof of Lemma~\ref{lem:Q:M1-M1-map}]       
        First, suppose $X\in\S(0,T)$. 
        Then there exist $m\in\mathbb{N}$, $0=t_0<\ldots<t_m=T$, and $\xi_k\in\F{t_k}$, $0\leq k\leq m-1$ such that
        \begin{equation*}
            X=\sum_{k=0}^{N-1} \xi_k\ind_{\l[t_k,t_{k+1}\r)}.
        \end{equation*}
        Analogous to the proof of Lemma~\ref{lem:I:M2-M2-map}, we have
        \begin{align*}
            Z &=
            \sum_{k=0}^{m-1} \xi_k \l(\l<B^a\r>_\cdot - \l<B^a\r>_{t_k} \r)\ind_{[t_k,t_{k+1})} + \sum_{k=0}^{N-1} \xi_k\l(\l<B^a\r>_{t_{k+1}} - \l<B^a\r>_{t_k}\r)\ind_{[t_{k+1},T)}
            \numberthis\label{eq:M1-M1-1}
        \end{align*}
        For the former sum on the right-hand side of \eqref{eq:M1-M1-1}, note that
        \begin{equation*}
            \xi_k \l(\l<B^a\r>_\cdot - \l<B^a\r>_{t_k}\r)\ind_{[t_k,t_{k+1})} = \xi_k\ind_{[t_k,t_{k+1})}\cdot \l(\l<B^a\r>_\cdot - \l<B^a\r>_{t_k}\ind_{[t_k,t_{k+1})}\r).
        \end{equation*}
        Clearly, $\l<B^a\r>_{t_k}\in\L^1(t_{k})$. 
        Thus $\l<B^a\r>_{t_k}\ind_{[t_k,t_{k+1})}\in\M^1(0,T)$ due to Lemma~\ref{lem:Lp-sum-Mp}.
        By Lemma~\ref{lem:B-M2}, we have $\l<B^a\r>\in\M^1(0,T)$ and thus $(\l<B^a\r>-\l<B^a\r>_{t_k}\ind_{[t_k,t_{k+1})})\in\M^1(0,T)$.
        Since $\xi_k\ind_{[t_{k+1},T)}\in\S(0,T)$, we obtain $\xi_k \l(\l<B^a\r>_\cdot - \l<B^a\r>_{t_k}\r)\ind_{[t_k,t_{k+1})}\in\M^1(0,T)$ due to Corollary~\ref{cor:A3}.

        Similarly, for the latter sum on the right-hand side of \eqref{eq:M1-M1-1},
        \begin{equation*}
            \xi_k\l(\l<B^a\r>_{t_{k+1}} -\l<B^a\r>_{t_k}\r)\ind_{[t_{k+1},T)}
            =\xi_k\ind_{[t_{k+1},T)}\cdot \l(\l<B^a\r>_{t_{k+1}} -\l<B^a\r>_{t_k}\r)\ind_{[t_{k+1},T)}.
        \end{equation*}
        We have $\l<B^a\r>_{t_k},\l<B^a\r>_{t_{k+1}}\in\L^1(t_{k+1})$ and thus $(\l<B^a\r>_{t_{k+1}} -\l<B^a\r>_{t_k})\ind_{[t_{k+1},T)} \in\M^1(0,T)$ due to Lemma~\ref{lem:Lp-sum-Mp}. 
        Corollary~\ref{cor:A3} yields $\xi_k(\l<B^a\r>_{t_{k+1}} -\l<B^a\r>_{t_k})\ind_{[t_{k+1},T)}\in\M^1(0,T)$.

        Finally, we deduce that $Z\in\M^1(0,T)$ as finite sum of elements in $\M^1(0,T)$.

        Now, suppose $X\in\M^1(0,T)$.
        Then there exists a sequence $(X^m)_{m\in\Nbb}$ in $\S(0,T)$ with
        \begin{equation*}
            \lim_{m\rightarrow\infty}\l\|X-X^m\r\|_{\M^1} =0.
        \end{equation*}
        For each $m\in\Nbb$, define $Z^m$ by
        \begin{equation*}
            Z^m_t:=\int_0^t X^m_s \d \l<B^a\r>_s
            =\CQ_a\l(X^m\ind_{[0,t)}\r) 
            ,\qquad 0\leq t\leq T.
        \end{equation*}
        Then $Z^m\in\M^1(0,T)$.
        By Lemma~\ref{lem:Q-properties}, we have
        \begin{equation*}
            \E\l[ \l| Z^m_t - Z_t \r|\r] 
            \leq \overline{\sigma}_{aa}^2 \l\| \l(X^m - X\r)\ind_{[0,t)} \r\|_{\M^1} 
            \leq \overline{\sigma}_{aa}^2 \l\| X^m - X \r\|_{\M^1},
        \end{equation*}
        and thus
        \begin{equation*}
            \lim_{m\rightarrow\infty} \int_0^T \E\l[ \l| Z^m_t - Z_t \r|^1\r] \d t
            \leq \lim_{m\rightarrow\infty} \overline{\sigma}_{aa}^2\,T \l\| X^m - X \r\|_{\M^1}
            = 0.
        \end{equation*}
        Hence, $Z\in\M^1(0,T)$ since $\M^1(0,T)$ is complete.
    \end{proof}

    \begin{proof}[Proof of Lemma~\ref{lem:int:M1-M1-map}]       
        First, suppose $X\in\S(t,T)$. 
        Then there exist $m\in\mathbb{N}$, $t=t_0<\ldots<t_m=T$, and $\xi_k\in\F{t_k}$, $0\leq k\leq m-1$ such that
        \begin{equation*}
            X=\sum_{k=0}^{m-1} \xi_k\ind_{\l[t_k,t_{k+1}\r)}.
        \end{equation*}
        Analogous to the proof of Lemma~\ref{lem:I:M2-M2-map}, we have
        \begin{align*}
            Z &=
            \sum_{k=0}^{m-1} \xi_k \l(\textnormal{id}_{[0,T]} - {t_k}\r)\ind_{[t_k,t_{k+1})} + \sum_{k=0}^{m-1} \xi_k\l({t_{k+1}} - {t_k}\r)\ind_{[t_{k+1},T)}.
            \numberthis\label{eq:A6-1}
        \end{align*}
        For the former sum on the right-hand side of \eqref{eq:A6-1}, note that
        \begin{equation*}
            \xi_k \l(\textnormal{id}_{[t,T]} - {t_k}\r)\ind_{[t_k,t_{k+1})} = \xi_k\ind_{[t_k,t_{k+1})}\cdot \l(\textnormal{id}_{[t,T]} - {t_k}\ind_{[t_k,t_{k+1})}\r).
        \end{equation*}
        Clearly, ${t_k}\ind_{[t_k,t_{k+1})}\in\S(0,T)$ since $t_k$ is constant.
        By Lemma~\ref{lem:t-Mp}, we have $\textnormal{id}_{[t,T]}\in\M^p(t,T)$ and thus $(\textnormal{id}_{[t,T]}-{t_k}\ind_{[t_k,t_{k+1})})\in\M^p(t,T)$.
        Since $\xi_k\ind_{[t_{k+1},T)}\in\S(t,T)$, we obtain $\xi_k \l(\textnormal{id}_{[t,T]}-{t_k}\ind_{[t_k,t_{k+1})}\r)\ind_{[t_k,t_{k+1})}\in\M^p(t,T)$ due to Corollary~\ref{cor:A3}.

        Similarly, for the latter sum on the right-hand side of \eqref{eq:M1-M1-1},
        \begin{equation*}
            \xi_k\l(\l<B^a\r>_{t_{k+1}} -\l<B^a\r>_{t_k}\r)\ind_{[t_{k+1},T)}
            =\xi_k\ind_{[t_{k+1},T)}\cdot \l(\l<B^a\r>_{t_{k+1}} -\l<B^a\r>_{t_k}\r)\ind_{[t_{k+1},T)}.
        \end{equation*}
        We have $\l<B^a\r>_{t_k},\l<B^a\r>_{t_{k+1}}\in\L^1(t_{k+1})$ and thus $(\l<B^a\r>_{t_{k+1}} -\l<B^a\r>_{t_k})\ind_{[t_{k+1},T)} \in\M^1(0,T)$ due to Lemma~\ref{lem:Lp-sum-Mp}. 
        Corollary~\ref{cor:A3} yields $\xi_k(\l<B^a\r>_{t_{k+1}} -\l<B^a\r>_{t_k})\ind_{[t_{k+1},T)}\in\M^1(0,T)$.

        Finally, we deduce that $Z\in\M^1(0,T)$ as finite sum of elements in $\M^1(0,T)$.

        Now, suppose $X\in\M^1(0,T)$.
        Then there exists a sequence $(X^m)_{m\in\Nbb}$ in $\S(0,T)$ with
        \begin{equation*}
            \lim_{m\rightarrow\infty}\l\|X-X^m\r\|_{\M^1} =0.
        \end{equation*}
        For each $m\in\Nbb$, define $Z^m$ by
        \begin{equation*}
            Z^m_t:=\int_0^t X^m_s \d \l<B^a\r>_s
            =\CQ_a\l(X^m\ind_{[0,t)}\r) 
            ,\qquad 0\leq t\leq T.
        \end{equation*}
        Then $Z^m\in\M^1(0,T)$.
        By Lemma~\ref{lem:Q-properties}, we have
        \begin{equation*}
            \E\l[ \l| Z^m_t - Z_t \r|\r] 
            \leq \overline{\sigma}_{aa}^2 \l\| \l(X^m - X\r)\ind_{[0,t)} \r\|_{\M^1} 
            \leq \overline{\sigma}_{aa}^2 \l\| X^m - X \r\|_{\M^1},
        \end{equation*}
        and thus
        \begin{equation*}
            \lim_{m\rightarrow\infty} \int_0^T \E\l[ \l| Z^m_t - Z_t \r|^1\r] \d t
            \leq \lim_{m\rightarrow\infty} \overline{\sigma}_{aa}^2\,T \l\| X^m - X \r\|_{\M^1}
            = 0.
        \end{equation*}
        Hence, $Z\in\M^1(0,T)$ since $\M^1(0,T)$ is complete.
    \end{proof}

    \begin{proof}[Proof of Lemma~\ref{lem:Q-p-bound}]
        Let $P\in\CP$, then $B$ is a continuous martingale under $P$. 
        By the martingale representation theorem, there exists $P$-Brownian motion $W^P$ and a progressively measurable $\sigma^P$ such that $B=\sigma^P \bullet W^P$ under $P$.
        Moreover, we can choose $\sigma^P$ to be $\Sigma$-valued due to the construction of $\CP$.
        
        By the definition of the quadratic covariation, we have for all $0\leq t\leq T$ under $P$ that
        \begin{align*}
            \l<B^a,B^b\r>_t
            &= 
            \frac{1}{4}\l( \l<B^{a+b}\r>_t - \l<B^{a-b}\r>_t \r)
            \\&=
            \frac{1}{4}\l( \l<\l(a+b\r)^TB\r>_t - \l<\l(a-b\r)^TB\r>_t \r)
            \\&= 
            \frac{1}{4}\l( \l<\l(a+b\r)^T\sigma^P\bullet W^P\r>_t - \l<\l(a-b\r)^T\sigma^P\bullet W^P\r>_t \r)
            \\&=
            \frac{1}{4}\l( \int_0^t \l(a+b\r)^T\sigma^P_s \d \l<W^P\r>_s\sigma^P_s\l(a+b\r) - \int_0^t \l(a-b\r)^T\sigma^P_s \d \l<W^P\r>_s\sigma^P_s\l(a-b\r) \r)
            \\&=
             \frac{1}{4} \int_0^t \l(a+b\r)^T\sigma^P_s \sigma^P_s\l(a+b\r) - \l(a-b\r)^T\sigma^P_s\sigma^P_s\l(a-b\r) \d s
             \\&=
             \int_0^t a^T\sigma^P_s\sigma^P_s b \d s.
        \end{align*}
        Using the representation of $\E$ as upper expectation, we obtain
        \begin{align*}
            \E\l[ \sup_{t\leq w\leq s}\l|\int_t^w X_u \d \l<B^a,B^b\r>_u\r|^p \r]
            &=
            \sup_{P\in\CP} E_P\l[ \sup_{t\leq w\leq s}\l|\int_t^w X_u a^T \sigma^P_u \sigma^P_u b \d u\r|^p \r]
            \\&\leq 
            \l(s-t\r)^{p-1} \sup_{P\in\CP} E_P\l[ \int_t^s \l| X_u a^T \sigma^P_u \sigma^P_u b \r|^p \d u\r]
            \\&\leq 
            \overline{\sigma}_{ab}^{2p} \l(s-t\r)^{p-1} \sup_{P\in\CP} E_P\l[ \int_t^s \l| X_u \r|^p \d u\r]
            \\&= 
            \overline{\sigma}_{ab}^{2p} \l(s-t\r)^{p-1} \E\l[ \int_t^s  \l| X_u \r|^p \d u \r] 
            \\&\leq 
            \overline{\sigma}_{ab}^{2p} \l(s-t\r)^{p-1} \int_t^s \E\l[  \l| X_u \r|^p \r] \d u.
        \end{align*}
    \end{proof}

    \begin{proof}[Proof of Lemma~\ref{lem:I-p-bound}]
        Let $P\in\CP$, then $X\bullet B$ is a continuous local martingale. 
        By the Burkholder-Davis-Gundy inequality, there exists a constant $C_p>0$ such that
        \begin{align*}
            \E\l[ \sup_{t\leq w\leq s} \l|\int_t^w X_u \d B^a_u\r|^p \r] 
            &= \sup_{P\in\CP} E_P\l[ \sup_{t\leq w\leq s} \l|\int_t^w X_u \d B^a_u\r|^p \r] 
            \\&\leq 
            C_p \sup_{P\in\CP} E_P\l[ \sup_{t\leq w\leq s} \l| \int_t^w X_u^2 \d\l< B^a\r>_u\r|^{\frac{p}{2}} \r]
            \\&=
            C_p \E\l[ \sup_{t\leq w\leq s} \l| \int_t^w X_u^2 \d\l< B^a\r>_u\r|^{\frac{p}{2}} \r]
            \\&\leq 
            C_p \,\overline{\sigma}^{p}_{aa} \l(s-t\r)^{\frac{p-2}{2}} \int_t^s \E\l[ \l|X_u \r|^p \r] \d u,
        \end{align*}
        where the last step follows from Lemma~\ref{lem:Q-p-bound} since $X^2\in\M^{q}(0,T)$ with $q=\frac{p}{2}\geq 1$.
     \end{proof}
    
    \printbibliography
\end{document}